\documentclass{amsart}
\usepackage{ amsmath, amsthm, amsfonts, hyperref, graphicx, color, ifpdf}

\newcommand{\ol}{\overline}

\newcommand{\be}{\begin{equation}}
\newcommand{\ee}{\end{equation}}
\newcommand{\tf}{\tilde{f}}
\newcommand{\tu}{\tilde{u}}
\newcommand{\lll}{\left}
\newcommand{\rrr}{\right}

\newcommand{\tgij}{\tilde{g}_{ij}}
\newcommand{\thij}{\tilde{h}_{ij}}
\newcommand{\tguij}{\tilde{g}^{ij}}

\newcommand{\tgukl}{\tilde{g}^{kl}}
\newcommand{\thh}{\tilde{h}}
\newcommand{\tg}{\tilde{g}}
\newcommand{\F}{(F-\sigma)}
\newcommand{\tna}{\tilde{\nabla}}
\newcommand{\na}{\nabla}
\newcommand{\ka}{\kappa}

\newtheorem{theorem}{Theorem}[section]
\newtheorem{lemma}[theorem]{Lemma}
\newtheorem{proposition}[theorem]{Proposition}
\newtheorem{corollary}[theorem]{Corollary}

\theoremstyle{definition}

\theoremstyle{remark}
\newtheorem{remark}[theorem]{Remark}

\numberwithin{equation}{section}

\begin{document}
\setlength{\baselineskip}{1.2\baselineskip}

\title[Curvature flow in hyperbolic space]
{Curvature flow of complete convex hypersurfaces in hyperbolic space}

\author{Ling Xiao}




\begin{abstract}
We investigate the existence, convergence and uniqueness of modified general curvature flow ($\mathbf{MGCF}$) of convex hypersurfaces in hyperbolic space with a prescribed asymptotic boundary.
\end{abstract}

\maketitle

\section{Introduction} \label{Int}
In this paper, we continue our study of modified curvature flow problems in hyperbolic space. Consider a complete (locally strictly) convex hypersurface in $\mathbb{H}^{n+1}$ with a prescribed asymptotic boundary $\Gamma$ at infinity, whose principal curvature is greater than $\sigma$ (e.g in our earlier work \cite{LX10} section 8 we gave an example of such "good" initial surfaces.)
and is given by an embedding $\mathbf{X}(0): \Omega\rightarrow\mathbb{H}^{n+1}$, where $\Omega\subset\partial_{\infty}\mathbb{H}^{n+1}.$ We consider the evolution of such embedding to produce a family of embeddings $\mathbf{X}:\Omega\times[0,T)\rightarrow\mathbb{H}^{n+1}$ satisfying the following equations
\begin{equation}{\label{Int0}}
\left\{
\begin{aligned}
&\dot{\mathbf{X}}=(f(\ka[\Sigma(t)])-\sigma)\nu_H&\,\,(x,t)\in\Omega\times[0,T),\\
&\mathbf{X}(0)=\Sigma_0&\,\,(x,t)\in\partial\Omega\times\{0\},\\
&\mathbf{X}=\Gamma&\,\,(x,t)\in\partial\Omega\times[0,T),
\end{aligned}
\right.
\end{equation}
where $\ka[\Sigma(t)]=(\ka_1,\cdots,\ka_n)$ denotes the hyperbolic principal curvatures of $\Sigma(t),$ $\sigma\in (0,1)$ is a constant and $\nu_{H}$ denotes the outward unit normal of $\Sigma(t)$ with respect to the hyperbolic metric.

In this paper, we shall use the half-space model,
\[\mathbb{H}^{n+1}=\{(x,x_{n+1})\in\mathbb{R}^{n+1}:x_{n+1}>0\}\]
equipped with the hyperbolic metric
\be\label{Int4}
ds^2=\frac{\sum_{i=1}^{n+1}dx_i^2}{x_{n+1}^2}.
\ee
One identifies the hyperplane $\{x_{n+1}=0\}=\mathbb{R}^n\times\{0\}\subset\mathbb{R}^{n+1}$ as infinity of
$\mathbb{H}^{n+1},$ denoted by $\partial_{\infty}\mathbb{H}^{n+1}.$ For convenience we say $\Sigma$ has compact asymptotic boundary if $\partial\Sigma\subset\partial_{\infty}\mathbb{H}^{n+1}$ is compact with respect to the Euclidean metric in $\mathbb{R}^n.$

The function $f$ is assumed to satisfy the following fundamental structure conditions:
\be\label{Int5}
f_i(\lambda)\equiv\frac{\partial f(\lambda)}{\partial\lambda_i}>0\;\;\mbox{in $K$},\;\;1\leq i\leq n,
\ee
\be\label{Int6}
\mbox{$f$ is a concave function in $K$},
\ee
and
\be\label{Int7}
f>0\;\;\mbox{in $K$},\;\;f=0\;\;\mbox{on $\partial K$},
\ee
 where $K\subset\mathbb{R}^n$ is an open symmetric convex cone defined as following
 \be\label{Int8}
 K:=K^+_n:=\{\lambda\in\mathbb{R}^n:\mbox{each component $\lambda_i>0$}\}.
 \ee
 In addition, we shall assume that $f$ is normalized
 \be\label{Int9}
 f(1,\cdots,1)=1
 \ee
 and satisfies more technical assumptions
 \be\label{Int10}
 \mbox{$f$ is homogeneous of degree one}.
 \ee
 Moreover,
 \be\label{Int11}
 \lim_{R\rightarrow +\infty}f(\lambda_1,\cdots,\lambda_{n-1},\lambda_{n}+R)\geq 1+\epsilon_0\;\;\;\mbox{uniformly in $B_{\delta_0}(\mathbf{1})$}
 \ee
 for some fixed $\epsilon_0>0$ and $\delta_0>0,$ where $B_{\delta_0}(\mathbf{1})$ is the ball of radius $\delta_0$ centered at $\mathbf{1}=(1,\cdots,1)\in\mathbb{R}^n.$

 As shown in \cite{GS10}, an example of the function satisfies all assumptions above is given by
 $f=(H_n/H_l)^{\frac{1}{n-l}},\;\;\mbox{$0\leq l<n$},$ defined in $K,$ where $H_l$ is the normalized $l-th$ elementary symmetric polynomial. (e.g., $H_0=1,$ $H_1=H,$ $H_n=K$ the extrinsic Gauss curvature.) 
 
 Since $f$ is symmetric, by (\ref{Int6}), (\ref{Int9}) and (\ref{Int10}) we have
 \be\label{Int12}
 f(\lambda)\leq f(\mathbf{1})+\sum f_i(\mathbf{1})(\lambda_i-1)=\sum f_i(\mathbf{1})\lambda_i=\frac{1}{n}\sum\lambda_i\;\;\mbox{in $K$}
 \ee
and
\be\label{Int13}
\sum f_i(\lambda)=f(\lambda)+\sum f_i(\lambda)(1-\lambda_i)\geq f(\mathbf{1})=1\;\;\mbox{in $K$.}
\ee

In this paper, we always assume the initial surfaces $\Sigma_0$ to be connected and orientable, and $\Sigma(t)=\{\mathbf{X}:=(x,u(x,t))\mid(x,t)\in\Omega\times[0,T),\;x_{n+1}=u(x,t)\}$ to be the flow surfaces with $\mathbf{X}=(x,u(x,t))$ satisfying the flow equation (\ref{Int0}). If $\Sigma$ is a complete hypersurface in $\mathbb{H}^{n+1}$ with compact asymptotic boundary at infinity, then the normal vector field of $\Sigma$ is always chosen to be the one pointing to the unique unbounded region in $\mathbb{R}^{n+1}_+/\Sigma,$ and both Euclidean and hyperbolic principal curvature of $\Sigma$ are calculated with respect to this normal vector field.

 We shall take $\Gamma=\partial\Omega,$ where $\Omega\subset\mathbb{R}^n$ is a smooth domain and seek a family of hypersurfaces $\Sigma(t)$ as a graph of function $u(x,t)$ with boundary $\Gamma.$ Then the coordinate vector fields and upper unit normal are given by
 \[\mathbf{X}_i=e_i+u_ie_{n+1},\;\nu_H=u\nu=u\frac{(-u_ie_i+e_{n+1})}{w},\]
 where through out this paper, $w=\sqrt{1+|\na u|^2}$ and $e_{n+1}$ is the unit vector in the positive $x_{n+1}$ direction in $\mathbb{R}^{n+1}.$

Note that by equation (\ref{Int0})
 \[\left<\dot{\mathbf{X}}, \nu_H\right>_H=f-\sigma,\]
which is equivalent to
\[\left<\frac{\partial}{\partial t}(x, u(x, t)), \nu_H\right>_H=f-\sigma,\]
from here we can derive that the height function $u$ satisfies equation
\be\label{Int1} u_t=(f-\sigma)uw.\ee

So problem (\ref{Int0}) then reduces to the Dirichlet problem for a fully nonlinear second order parabolic equation
\be\label{Int2}
\left\{
\begin{aligned}&u_t=uw(f-\sigma) &\mbox{on $\Omega\times [0, T)$\,,}\\
&u(x, 0)=u_0 &\mbox{on $\Omega\times \{0\}$\,,}\\
&u(x, t)=0 &\mbox{on $\partial\Omega\times [0, T)$\,.}
\end{aligned}
\right.
\ee

In this paper, we shall focus on proving the long time existence of the modified general curvature flow ($\mathbf{MGCF}$) of complete embeded hypersurfaces with initial surface whose principal curvature greater than $\sigma$ everywhere; furthermore, we shall also prove the uniqueness under additional assumptions.

To begin with, I'd like to state the following beautiful result of \cite{GSZ09}
\begin{theorem}\label{Intt0}
Let $\Sigma$ be a complete locally strictly convex $C^2$ hypersurface in $\mathbb{H}^{n+1}$ with compact asymptotic boundary at infinity. Then $\Sigma$ is the vertical graph of a function $u\in C^2(\Omega)\cap C^0(\ol{\Omega}),$ $u>0$ in $\Omega$ and $u=0$ on $\partial\Omega,$ for some domain $\Omega\subset\mathbb{R}^n:$
\[\Sigma=\{(x,u)\in\mathbb{R}^{n+1}_+:x\in\Omega\}\]
such that
\be\label{Int14}
\{\delta_{ij}+u_iu_j+uu_{ij}\}>0\;\;\mbox{in $\Omega$.}
\ee
That is, the function $u^2+|x|^2$ is strictly convex.
\end{theorem}

According to Theorem \ref{Intt0}, our assumption that $\Sigma(t)$ is a graph is completely general and the asymptotic boundary $\Gamma$ must be the boundary of some bounded domain $\Omega$ in $\mathbb{R}^n.$

We seek solution of equation (\ref{Int2}) satisfying (\ref{Int14}) for all $t\in[0,T)$. (We will see in section \ref{Pre} that when the initial surface of the MGCF under certain restriction then the solution of (\ref{Int2}) must satisfy (\ref{Int14}).) Following the literature we call such solutions $admissible.$ By \cite{CNS85} condition (\ref{Int5}) implies that equation (\ref{Int2}) is parabolic for admissible solutions.

The main result of this paper may be stated as follows.
\begin{theorem}\label{Intt1}
Let $\Gamma=\partial\Omega\times\{0\}\subset\mathbb{R}^{n+1}$ where $\Omega$ is a bounded smooth domain in $\mathbb{R}^n.$ Suppose that $\sigma\in(0,1)$ and that $f$ satisfies conditions (\ref{Int5})-(\ref{Int11}) with $K=K^+_n.$ Furthermore, let $\Sigma_0=\{(x,u_0(x))\mid u_0\in C^{\infty}(\Omega)\cap C^{1+1}(\ol{\Omega})\}$ be a complete locally strictly convex hypersurface with $\partial\Sigma_0=\Gamma$ and $f(\kappa[\Sigma_0])$ greater than $\sigma,$ then there exists a solution $\Sigma(t),\;t\in[0,\infty),$ to the MGCF (\ref{Int0}) with uniformly bounded principal curvatures
\be\label{Int15}
|\ka[\Sigma(t)]|\leq C\;\;\mbox{on $\Sigma(t),$ for all $t\in[0,\infty).$ }
\ee
Moreover, $\Sigma(t)=\{(x,u(x,t))\mid (x,t)\in\Omega\times[0,\infty)\}$ is the flow surfaces of an admissible solution $u(x,t)\in C^{\infty}(\Omega\times(0,\infty))\cap W^{2,1}_p(\Omega\times[0,\infty))$ of the Dirichlet problem (\ref{Int2}), where $p>4$. Furthermore, for any fixed $t>0,$ we have $u^2(x,t)\in C^{\infty}(\Omega)\cap C^{1+1}(\ol{\Omega})$ and
\be\label{Int16}
u|D^2u|\leq C\;\;\mbox{in $\Omega$,}
\ee
\be\label{Int17}
\sqrt{1+|Du|^2}\leq C\;\;\mbox{in $\Omega$,}
\ee
where $C$ is some constant independent of $t$. In addition, if
\be\label{Int20}
\sum f_i>\sum\lambda_i^2f_i\;\;\mbox{in $K\cap\{0<f<1\},$}
\ee
then as $t\rightarrow\infty,$ $u(t)$ converges uniformly to a function $\tilde{u}\in C^{\infty}(\Omega)\cap C^1(\ol{\Omega}),$ such that $\Sigma_{\infty}=\{(x,\tilde{u})\in\mathbb{R}^{n+1}, x\in\Omega\}$ is a unique complete locally strictly convex surface satisfies $f(\ka[\Sigma_{\infty}])=\sigma$ in $\mathbb{H}^{n+1}.$
\end{theorem}

Due to the degeneracy of equation (\ref{Int2}) when $u=0,$ it is very natural to consider the
approximate modified general curvature flow ($\mathbf{AMGCF}$) problem. Instead of $u=0\;\;\mbox{on $\partial{\Omega}$}$ one assumes $u=\epsilon\;\;\mbox{on $\partial{\Omega}$}$, $\epsilon$ is small enough. So the equations become,
\be\label{Int3}
\left\{
\begin{aligned}
&u_t=uw(f-\sigma)&\mbox{on $\Omega\times [0, T),$}\\
&u(x, 0)=u_0^\epsilon&\mbox{on $\Omega\times \{0\},$}\\
&u(x, t)=\epsilon&\mbox{on\, $\partial\Omega\times [0, T).$}
\end{aligned}
\right.
\ee
where $u_0^\epsilon=u_0+\epsilon$ and $\Sigma_0=\{(x,u_0^\epsilon)|x\in\Omega\}$ satisfies $f(\ka[\Sigma_0^\epsilon])>\sigma,\;\;\forall x\in\Omega.$

\begin{theorem}\label{Intt2}
Let $\Omega$ be a bounded smooth domain in $\mathbb{R}^n$ and $\sigma\in(0,1).$ Suppose $f$ satisfies (\ref{Int5})-(\ref{Int11}) with $K=K^+_n.$ Then for any $\epsilon>0$ sufficiently small, there exists an admissible solution $u^{\epsilon}\in C^{\infty}(\ol{\Omega}\times(0,\infty))$ of the Dirichlet Problem (\ref{Int3}).
Moreover, $u^{\epsilon}$ satisfies the a priori estimates
\be\label{Int18}
\sqrt{1+|Du^{\epsilon}|^2}\leq \frac{1}{\sigma}+C\epsilon,\;\;
u^{\epsilon}|D^2u^{\epsilon}|\leq C\;\;\mbox{on $\partial\Omega\times[0,\infty),$}
\ee
and
\be\label{Int19}
u^{\epsilon}|D^2u^{\epsilon}|\leq C(t,\epsilon)\;\;\mbox{in $\Omega\times[0,\infty).$}
\ee
In particular, $C(t,\epsilon)$ depends exponentially on time $t.$
\end{theorem}
\begin{remark}\label{Intr0}
The a priori estimates (\ref{Int18}) will be proved in section \ref{Gre} and \ref{C2b}, while estimate (\ref{Int19}) can be proved by combining Lemma \ref{c2gl0} and equation (\ref{c2g22}) then use standard maximum principle for parabolic equation.
\end{remark}

The main technical difficulty in proving Theorem \ref{Intt1} is that we can not use the estimates (\ref{Int19}) to pass to the limit. We overcome this difficulty by proving a maximum principle for the largest hyperbolic principal curvature.

\begin{theorem}\label{Intt3}
Let $\Omega$ be a bounded smooth domain in $\mathbb{R}^n$ and $\sigma\in(0,1).$ Suppose $f$ satisfies (\ref{Int5})-(\ref{Int11}) with $K=K^+_n.$ Then for any admissible solution $u^\epsilon$ of the Dirichlet problem (\ref{Int3}),
\be\label{Int21}
u^{\epsilon}|D^2u^{\epsilon}|\leq C(1+\max_{\partial\Omega\times[0,\infty)}u^{\epsilon}|D^2u^{\epsilon}|)
\;\;\mbox{in $\Omega\times[0,\infty)$,}
\ee
where $C$ is independent of $\epsilon$ and $t.$
\end{theorem}

By applying Theorem \ref{Intt3} to Theorem \ref{Intt2}, one can see that the hyperbolic curvatures of the admissible solution $u^{\epsilon}$ are uniformly bounded from above. Later we will also show that, if our initial surface satisfies $f(\ka[\Sigma_0])>\sigma$ then $f>\sigma$ during the flow process. In particular,

\begin{theorem}\label{Intt4}
Suppose $f$ satisfies (\ref{Int5})-(\ref{Int11}) with $K=K^+_n,$ and $u^{\epsilon}(x,t)$ is an admissible solution
of the Dirichlet problem (\ref{Int3}), and in addition 
\be\label{Int23}
f(\ka[\Sigma_0^\epsilon])>\sigma.
\ee
Then we have
\be\label{Int22}
f(\ka[\Sigma^\epsilon(t)])>\sigma\;\;\mbox{$\forall t\in[0,T).$}
\ee
\end{theorem}

Thus one can conclude that the hyperbolic curvatures admit a uniform positive lower bound, so by the interior estimates of Evans and Krylov, we obtain a uniform $C^{2,\alpha}$ estimates for any compact subdomain of $\Omega.$ Then the proof of Theorem \ref{Intt1} becomes routine.

The paper is organized as follows. In section \ref{Foh} we establish some basic identities for hypersurfaces in $\mathbb{H}^{n+1}.$ In section \ref{Sho} we prove the short time existence of the Dirichlet problem for AMGCF. Section \ref{Evo} contains some essential identities and evolution equations which will be used later. The preserving of convexity will be proved in section \ref{Pre}. Section \ref{Gre} contains a global gradient estimate, while in sections \ref{C2b} and \ref{c2g} we prove the boundary and global estimates for the second derivative of $u$ respectively. Finally in sections \ref{Con} and \ref{Unf}, we discuss the  convergence and uniqueness of the MGCF.

\section{Formulas for hyperbolic principal curvatures}\label{Foh}
\subsection{Formulas on hypersurfaces}\label{Fohfirst}
 We will compare the induced hyperbolic and Euclidean metrics and derive some basic identities on a hypersurface.

Let $\Sigma$ be a hypersurface in $\mathbb{H}^{n+1}.$ We shall use $g,$ and $\na$ to denote the induced hyperbolic metric and Levi-Civita connections on $\Sigma,$ respectively. Since $\Sigma$ also can be viewed as a submanifold of $\mathbb{R}^{n+1},$ we shall usually distinguish a geodesic quantity with respect to Euclidean metric by adding a 'tilde' over the corresponding hyperbolic quantity. For instance, $\tg$ denotes the induced metric on $\Sigma$ from $\mathbb{R}^{n+1},$ and $\tna$ is its Levi-Civita connection.

Let $(z_1,\cdots,z_n)$ be local coordinates and
$$\tau_i=\frac{\partial}{\partial z_i},\;\;\mbox{$i=1,\cdots,n$}.$$
The hyperbolic and Euclidean metrics of $\Sigma$ are given by
\be\label{Foh0}
g_{ij}=\lll<\tau_i, \tau_j\rrr>_H,\;\;\tg_{ij}=\tau_i\cdot\tau_j=u^2g_{ij},
\ee
while the second fundamental forms are
\be\label{Foh1}
\begin{aligned}
&h_{ij}=\lll<D_{\tau_i}\tau_j, \mathbf{n}\rrr>_H=-\lll<D_{\tau_i}\mathbf{n},\tau_j\rrr>_H,\\
&\thh_{ij}=\nu\cdot\tilde{D}_{\tau_i}\tau_j=-\tau_j\cdot\tilde{D}_{\tau_i}\nu,\\
\end{aligned}
\ee
where $D$ and $\tilde{D}$ denote the Levi-Civita connection of $\mathbb{H}^{n+1}$ and $\mathbb{R}^{n+1},$ respectively.
The following relations are well known (see equation(1.5),(1.6) of \cite{GS08} ):
\be\label{Foh2}
h_{ij}=\frac{1}{u}\thh_{ij}+\frac{\nu^{n+1}}{u^2}\tg_{ij}.
\ee
\be\label{Foh3}
\ka_i=u\tilde{\ka}_i+\nu^{n+1},\;\;\mbox{$i=1,\cdots,n,$}
\ee
where $\nu^{n+1}=\nu\cdot e_{n+1}=\frac{1}{w}.$

The Christoffel symbols are related by formula
\be\label{Foh9}
\Gamma^k_{ij}=\tilde{\Gamma}^k_{ij}-\frac{1}{u}(u_i\delta_{kj}+u_j\delta_{ik}-\tg^{kl}u_l\tg_{ij}).
\ee
It follows that for $v\in C^2(\Sigma)$
\be\label{Foh10}
\na_{ij}v=v_{ij}-\Gamma^k_{ij}v_k=\tna_{ij}v+\frac{1}{u}(u_iv_j+u_jv_i-\tg^{kl}u_kv_l\tg_{ij})
\ee
where and in the sequel (if no additional explanation)
\[v_i=\frac{\partial v}{\partial x_i},\;v_{ij}=\frac{\partial^2v}{\partial x_i\partial x_j},\;etc.\]
In particular,
\be\label{Foh11}
\na_{ij}u=\tna_{ij}u+\frac{2u_iu_j}{u}-\frac{1}{u}\tg^{kl}u_ku_l\tg_{ij}.
\ee
Moreover in $\mathbb{R}^{n+1},$
\be\label{Foh12}
\tg^{kl}u_ku_l=|\tna u|^2=1-(\nu^{n+1})^2
\ee
\be\label{Foh13}
\tna_{ij}u=\thh_{ij}\nu^{n+1}.
\ee
We note that all formulas above still hold for general local frame $\tau_1, \cdots, \tau_n.$ In particular, if $\tau_1, \cdots, \tau_n$ are orthonormal  in the hyperbolic metric, then $g_{ij}=\delta_{ij}$ and $\tg_{ij}=u^2\delta_{ij}.$

We now consider equation (\ref{Int0}) on $\Sigma.$ For $K$ as in section 1, let $\mathcal{A}$ be the vector space of $n\times n$ matrices and
\[\mathcal{A}_K=\left\{A=\{a_{ij}\}\in\mathcal{A}: \lambda(A)\in K\right\},\]
where $\lambda(A)=(\lambda_1,\cdots,\lambda_n)$ denotes the eigenvalues of $A.$ Let $F$ be the function defined by
\be\label{Foh4}
F(A)=f(\lambda(A)),\;\;A\in\mathcal{A}_K
\ee
and denote
\be\label{Foh5}
F^{ij}(A)=\frac{\partial F}{\partial a_{ij}}(A),\;\;F^{ij,kl}(A)=\frac{\partial^2F}{\partial a_{ij}\partial a_{kl}}(A).
\ee
Since $F(A)$ depends only on the eigenvalues of $A,$ if $A$ is symmetric then so is the matrix $\left\{F^{ij}(A)\right\}.$ Moreover,
\[F^{ij}(A)=f_i\delta_{ij}\]
when $A$ is diagonal, and
\be\label{Foh6}
F^{ij}(A)a_{ij}=\sum f_i(\lambda(A))\lambda_i=F(A),
\ee
\be\label{Foh7}
F^{ij}(A)a_{ik}a_{jk}=\sum f_i(\lambda(A))\lambda^2_i.
\ee
Equation (\ref{Int2}) can therefore be rewritten in a local frame $\tau_1,\cdots, \tau_n$ in the form
\be\label{Foh8}
\left\{
\begin{aligned}
&u_t=uw(F(A[\Sigma])-\sigma)&\,\,(x,t)\in\Omega\times[0,T),\\
&u(x,0)=u_0&\,\,(x,t)\in\Omega\times\{0\},\\
&u(x,t)=0&\,\,(x,t)\in\partial\Omega\times[0,T),
\end{aligned}
\right.
\ee
where $A[\Sigma]=\lll\{g^{ik}h_{kj}\rrr\}.$ let $F^{ij}=F^{ij}\lll(A[\Sigma]\rrr),$ $F^{ij,kl}=F^{ij,kl}\lll(A[\Sigma]\rrr).$
\subsection{Vertical graphs}\label{Fohsecond} Suppose $\Sigma$ is locally represented as the graph of a function $u\in C^2(\Omega),$ $u>0,$ in a domain $\Omega\subset\mathbb{R}^n:$
\[\Sigma=\{(x,u(x))\in\mathbb{R}^{n+1}: x\in\Omega\}.\]
In this case we take $\nu$ to be the upward (Euclidean) unit normal vector field to $\Sigma:$
\[\nu=\lll(-\frac{Du}{w},\frac{1}{w}\rrr),\;w=\sqrt{1+|Du|^2}.\]
The Euclidean metric and second fundamental form of $\Sigma$ are given respectively by
\[\tg_{ij}=\delta_{ij}+u_iu_j,\]
and
\[\thh_{ij}=\frac{u_{ij}}{w}.\]
According to \cite{CNS86}, the Euclidean principal curvature $\tilde{\kappa}[\Sigma]$ are the eigenvalues of symmetric matrix $\tilde{A}[u]=[\tilde{a}_{ij}]:$
\be\label{Foh14}
\tilde{a}_{ij}:=\frac{1}{w}\gamma^{ik}u_{kl}\gamma^{lj},
\ee
where
\[\gamma^{ij}=\delta_{ij}-\frac{u_iu_j}{w(1+w)}.\]
Note that the matrix $\{\gamma^{ij}\}$ is invertible with the inverse
\[\gamma_{ij}=\delta_{ij}+\frac{u_iu_j}{1+w}\]
which is the square root of $\{\tilde{g}_{ij}\},$ i.e., $\gamma_{ik}\gamma_{kj}=\tilde{g}_{ij}.$ From (\ref{Foh3}) we see that the hyperbolic principal curvatures $\kappa[u]$ of $\Sigma$ are eigenvalues of the matrix $A[u]=\{a_{ij}[u]\}:$
\be\label{Foh15}
a_{ij}:=\frac{1}{w}\lll(\delta_{ij}+u\gamma^{ik}u_{kl}\gamma^{lj}\rrr).
\ee
When $\Sigma$ is a vertical graph we can also define $F(A[\Sigma])=F(A[u]).$

\section{Short time existence} \label{Sho}
In order to prove a global existence for the Dirichlet problem (\ref{Int3}), first of all, we need start with short time existence.

We state a more general result of short time existence.

\begin{theorem}\label{Shol0}
Let $G(D^2u,Du,u)$ be a nonlinear operator, which is smooth with respect to $u, Du$ and $D^2u.$
Suppose that $G$ is defined for function $u$ belonging to an open set $\Lambda\subset C^2(\Omega)$
and $G$ is elliptic for any $u\in\Lambda,$ i.e., $G^{ij}>0.$ Then the initial value problem
\be\label{Sho0}
\left\{
\begin{aligned}
&u_t=G(D^2u,Du,u)\;\;&\mbox{on $\Omega\times [0, T),$}\\
&u(x, 0)=u_0 \;\;&\mbox{on $\Omega\times \{0\},$}\\
&u(x, t)=u_0|_{\partial\Omega}\;\; &\mbox{on $\partial\Omega\times [0, T),$}
\end{aligned}
\right.
\ee
has a unique solution $u$ for $T=\epsilon >0$ small enough. Furthermore, $u$ is smooth
except for the corner, when $u_0\in\Lambda$ is of class $C^{\infty}(\ol\Omega).$
\end{theorem}
\begin{proof}Here we will apply the inverse function theorem to prove the local existence of equation (\ref{Sho0}). Following the idea of \cite{C89} Section 3 (also \cite{G06} Chapter 2).

(i)Let $\tilde{u}$ be a solution of the linear parabolic problem
\be\label{Sho1}
\left\{
\begin{aligned}
&\dot{\tilde{u}}-\Delta\tilde{u}=G(D^2u_0,Du_0,u_0)-\Delta u_0\,,\\
&\tilde{u}(x,0)=u_0\,,\\
&\tilde{u}|_{\partial\Omega\times[0, T)}=u_0|_{\partial\Omega}\,.
\end{aligned}
\right.
\ee
By standard parabolic PDE theory (see \cite{H75} part IV pg.117 and \cite{C89} theorem 3.1),
we know that there exists a unique solution $\tilde{u}\in V\subset V_c,$
where $V=W^{2,1}_p (\Omega\times[0, T],  \mathbb{R})\bigcap C^{\infty}(\ol\Omega\times (0, T)),$
$V_c=W^{2,1}_p (\Omega\times[0, T], \mathbb{R})\bigcap C^{\infty}(\ol\Omega\times [c, T]), ${\footnote{$W^{2,1}_p$ is the space of function $f$ such that the norms
\[\|f\|_{w^{2,1}_p}=\|f\|_{Lp}+\left\|\frac{\partial f}{\partial t}\right\|_{Lp}+\sum_{i=1}^n\left\|\frac{\partial f}{\partial x_i}\right\|_{Lp}+\sum_{i,j=1}^n \left\|\frac{\partial^2f}{\partial x_i\partial x_j}\right\|_{Lp}\] are finite.}} $0<c<T$, $T>0$ is small enough  so that $\tilde{u}(\cdot, t)\in\Lambda$ for any $0\leq t\leq T,$ in addition, $4<p<\infty$.

(ii) Now let's define
\[\tf=\dot{\tu}-G(D^2\tu, D\tu, \tu).\]
It's easy to see that $\tf\in L^p(\Omega\times[0, T], \mathbb{R}),$
moreover there holds
\be\label{Sho2}
\tf(x, 0)=0,\,x\in\Omega.
\ee

(iii) Now consider the nonlinear operator $\Phi$
\[\Phi(u)=\left(\dot{u}-G(D^2u,Du,u), u|_{\Omega\times\{0\}}, u\mid_{\partial\Omega\times[0, T]}\right)\]
with image in
\[W=L^p\left(\Omega\times[0, T],\mathbb{R}\right)\times W^{2, 1}_p\lll(\Omega,\mathbb{R}\rrr)\times B^{2-1/p, 1-1/2p}_p\lll(\partial\Omega\times[0,T],\mathbb{R}\rrr).\] {\footnote{$B^{2-1/p,1-1/2p}_p$ is the space of functions f such that the norms
$$\begin{aligned}\|f\|_{B^{2-1/p,1-1/2p}_p}&=\|f\|_{Lp}+\left(\int_0^\infty\tau^{-1/2 -p}\left\|f(x,t)-f(x,t+\tau)\right\|^p_{Lp}d\tau\right)^{1/p}\\
&+\sum_{i=1}^n\left(\int_0^{\infty}\tau^{-1-p}\|2f(x,t)-f(x+\tau e_i,t)-f(x-\tau e_i,t)\|^p_{Lp}d\tau\right)^{1/p}\\
&+\sum_{i,j=1}^n\left(\int_0^\infty\tau^{-p}\|f_{x_j}(x+\tau e_i,t)-f_{x_j}(x,t)\|^p_{Lp}d\tau\right)^{1/p}\end{aligned}$$ are finite.
}}
It's clear that $\Phi$ is well defined in a neighborhood of $\tu\in V\subset V_c.$ Moreover,
$\Phi$ is continuously differentiable and its derivative $D\Phi$ evaluated at $\tu,$ equals the operator:
\[\mathfrak{L}\eta=\lll(\dot{\eta}- G^{ij}(\tu)\eta_{ij}-G^i(\tu)\eta_i-G^{u}(\tu)\eta,\eta|_{\Omega\times\{0\}},\eta|_{\partial\Omega\times[0, T]}\rrr)\]
for any $\eta\in V_c,$ and $Pr_1\circ\mathfrak{L}$ represents a uniformly parabolic linear operator with coefficients in $L^p\lll(\Omega\times[0,T], \mathbb{R}\rrr).$
Applying the inverse function theorem, we deduce that $\Phi$ restricted to a small ball $B_\rho(\tu)\subset V\subset V_c$ is a $C^1$--diffeomorphism onto an open neighborhood $U(\tf,u_0,u_0|_{\partial\Omega})\subset W.$
Now let
\be\label{Sho3}
\eta_\epsilon(t)=\left\{
\begin{aligned}&0,\,\,0\leq t\leq \epsilon\,,\\
&1,\,\,2\epsilon\leq t\leq T.
\end{aligned}
\right.
\ee
Define $f_\epsilon=\tf\eta_\epsilon,$ since $\lim_{\epsilon\rightarrow0}\lll|f_\epsilon-\tf\rrr|_{L^p( \Omega\times[0,T])}=0,$ when $\epsilon$ small enough we have
\[\lll(f_\epsilon,u_0,u_0|_{\partial\Omega}\rrr)\in U(\tf,u_0,u_0|_{\partial\Omega}).\]
Thus for $0\leq t\leq \epsilon$ there exists $u\in B_\rho(\tu)\subset V$ solves the initial value problem (\ref{Sho0}).

(iv) It remains to prove the uniqueness of the solution of equation (\ref{Sho0}).
If not, let $u$ and $\tu$ be two solutions in $V.$ We only need to show that when $0\leq t\leq\delta,\,0<\delta\leq\epsilon,$ $u$ and $\tu$ agree. The final result, that $u$ and $\tu$ agree on the whole interval $[0, \epsilon]$ follows by repeating the argument.

If $0<\delta$ is sufficiently small then the convex combination satisfies
\be\label{Sho4}
u_\tau=\tau u+(1-\tau)\tu\in\Lambda,\,\forall (t, \tau)\in [0,\delta]\times[0,1],
\ee
hence we deduce
\be\label{Sho5}
\begin{aligned}
0&=\dot{u}-\dot{\tu}-G(D^2u,Du,u)+G(D^2\tu,D\tu,\tu)\\
&=\int_0^1\frac{d}{d\tau}\lll[\dot{u}_\tau-G(D^2u_\tau,Du_\tau,u_\tau)\rrr]\\
&=(u-\tu)_t-a^{ij}(u-\tu)_{ij}-b^i(u-\tu)_i-c(u-\tu)
\end{aligned}
\ee
where $a^{ij}$ is positive definite. Since $u|_{\Omega\times\{0\}}=\tu|_{\Omega\times\{0\}}$ and $u|_{\partial\Omega\times[0,T)}=\tu|_{\partial\Omega\times[0,T)},$ the result $u=\tu$ in $[0,\delta]$ then follows from the parabolic maximum principle. (\cite{L96} Theorem 14.1, pg363)
\end{proof}

\section{Evolution equations for some geometric quantities}\label{Evo}
In this section, we will compute the evolution equations for some affine geometric quantities. Before we start, need to point out that in this section for $v\in C^2(\Sigma),$ we denote
$v_i=\tna_iv,$ $v_{ij}=\tna_{ij}v,$ etc.

\begin{lemma}\label{Evol0}(Evolution of the metrics). The metric $g_{ij}$ and $\tg_{ij}$ of $\Sigma(t)$ satisfies the evolution equations
\end{lemma}
\be\label{Evo8}\dot{g}_{ij}=-2u^{-2}\tgij(F-\sigma)w-2u^{-1}(F-\sigma)\thij,\ee
and
\be\label{Evo7}\dot{\tilde{g}}_{ij}=-2(F-\sigma)u\tilde{h}_{ij}.\ee
\begin{proof}
Since $\tgij=\tau_i\cdot\tau_j,$
\begin{align*}
&\frac{\partial}{\partial t}\tgij=2\lll<\tilde{D}_{\tau_i}\dot{X},\tilde{D}_{\tau_j}X\rrr>\\
&=2\lll<\tilde{D}_{\tau_i}[(F-\sigma)u\nu],\tau_j\rrr>\\
&=2(F-\sigma)u\lll<\tilde{D}_{\tau_i}\nu,\tau_j\rrr>\\
&=-2(F-\sigma)u\thij.
\end{align*}
Differentiating equation (\ref{Foh0}) with respect to $t$ we get
\begin{align*}
\frac{\partial}{\partial t}g_{ij}&=-2u^{-3}\tgij u_t+u^{-2}\dot{\tilde{g}}_{ij}\\
&=-2u^{-3}\tgij(F-\sigma)uw-2u^{-2}(F-\sigma)u\thij\\
&=-2u^{-2}\tgij(F-\sigma)w-2u^{-1}(F-\sigma)\thij.
\end{align*}
\end{proof}

\begin{lemma}\label{Evol1}(Evolution of the normal). The normal vector evolves according to
\be\label{Evo9} \dot{\nu}=-\tg^{ij}[(F-\sigma)u]_i\tau_j,\ee
moreover,
\be\label{Evo10}\dot{\nu}^{n+1}=-\tilde{g}^{ij}[(F-\sigma)u]_iu_j.\ee
\end{lemma}
\begin{proof}
Since $\nu$ is the unit normal vector of $\Sigma$, we have $\dot{\nu}\in T(\Sigma).$ Furthermore, differentiating
$$\lll<\nu,\tau_i\rrr>=\lll<\nu,\tilde{D}_{\tau_i}X\rrr>=0,$$
with respect to $t$ we deduce
\begin{align*}
\lll<\dot{\nu},\tau_i\rrr>&=-\lll<\nu,\tilde{D}_{\tau_i}[(F-\sigma)u\nu]\rrr>\\
&=-\lll<\nu,[(F-\sigma)u]_i\nu\rrr>\\
&=-[(F-\sigma)u]_i,
\end{align*}
so we have
\[\dot{\nu}=-\tguij[(F-\sigma)u]_i\tau_j.\]
Thus (\ref{Evo10}) follows directly from
\[\dot{\nu}^{n+1}=\lll<\dot{\nu},e_{n+1}\rrr>\; \mbox{and}\; u_j=\tau_j\cdot e_{n+1}.\]
\end{proof}

\begin{lemma}\label{Evol2}(Evolution of the second fundamental form). The second fundamental form evolves according to
\be\label{Evo11}\dot{\tilde{h}}^l_i=[(F-\sigma)u]^l_i+u(F-\sigma)\tilde{h}^k_i\tilde{h}^l_k,\ee
\be\label{Evo12}\dot{\tilde{h}}_{ij}=[(F-\sigma)u]_{ij}-u(F-\sigma)\tilde{h}^k_i\tilde{h}_{kj},\ee
and
\be\label{Evo13}\begin{aligned}&\dot{h}_{ij}=\frac{1}{u}\{[(F-\sigma)u]_{ij}-u(F-\sigma)\tilde{h}^k_i\tilde{h}_{kj}\}
-\frac{\thij}{u}w(F-\sigma)\\
&-\{\tg^{kl}[u(F-\sigma)]_ku_l\}\frac{\tgij}{u^2}
-2\frac{(F-\sigma)\nu^{n+1}}{u}\thij-2\frac{\tgij}{u^2}(F-\sigma).\end{aligned}\ee
\end{lemma}
\begin{proof}
Differentiating (\ref{Evo9}) with respect to $\tau_i$ we get
\[\frac{\partial}{\partial t}\nu_i=-\tgukl[(F-\sigma)u]_{ki}\tau_l
-\tgukl[(F-\sigma)u]_k\tilde{D}_{\tau_i}\tau_l.\]
On the other hand, in view of the Weingarten Equation
\[\nu_i=-\tg^{kl}\thh_{li}\tau_k\Rightarrow\dot{\nu_i}=-\dot{\thh}^k_i\tau_k-\thh^k_i\tilde{D}_{\tau_k}\dot{X},\]
where $\thh^k_i=\tg^{kl}\thh_{li}$ is mixed tensor,
multiply by $\tau_j$ we get
\[-\dot{\thh}^k_i\tg_{kj}-\thh^k_i\lll<\tilde{D}_{\tau_k}\dot{X},
\tau_j\rrr>=-\tgukl[(F-\sigma)u]_{ki}\tg_{lj}.\]
Therefore
\begin{align*}\dot{\thh}^k_i\tg_{kj}&=\tgukl[(F-\sigma)u]_{ki}\tg_{lj}-\thh^k_iu(F-\sigma)\lll<\tilde{D}_{\tau_k}\nu,
\tau_j\rrr>\\
&=[(F-\sigma)u]_{ij}+u(F-\sigma)\thh^k_i\thh_{kj}.\end{align*}
Multiplying the resulting equation with $\tg^{jl}$
\be\label{Evo0}
\dot{\thh}^l_i=[(F-\sigma)u]^l_i+u(F-\sigma)\thh^k_i\thh^l_k.
\ee
Moreover, since $\thij=\thh^l_i\tg_{lj},$ differentiating it with respect to $t$ and use equation (\ref{Evo7}) get
\begin{align*}
\dot{\thh}_{ij}&=\dot{\thh}^l_i\tg_{lj}+\thh^l_i\dot{\tg}_{lj}\\
&=[(F-\sigma)u]^l_i\tg_{lj}+u(F-\sigma)\thh^k_i\thh^l_k\tg_{lj}+\thh^l_i[-2(F-\sigma)u\thh_{lj}]\\
&=[(F-\sigma)u]_{ij}-u(F-\sigma)\thh^k_i\thh_{kj}.
\end{align*}
Finally by differentiating equation (\ref{Foh2}) with respect to $t$,
we have
\be\label{Evo1}
\begin{aligned}
&\frac{\partial}{\partial t}h_{ij}=\frac{1}{u}\dot{\thh}_{ij}-\frac{\thij}{u^2}u_t+\frac{\tgij}{u^2}\dot{\nu}^{n+1}
+\frac{\nu^{n+1}}{u^2}\dot{\tg}_{ij}-2\frac{\nu^{n+1}\tgij}{u^3}u_t\\
&=\frac{1}{u}\{[(F-\sigma)u]_{ij}-u(F-\sigma)\thh^k_i\thh_{kj}\}-\frac{\thij}{u}w(F-\sigma)\\
&+\frac{\tgij}{u^2}\{-\tg^{kl}[u(F-\sigma)]_ku_l\}+\frac{\nu^{n+1}}{u^2}[-2(F-\sigma)u\thij]
-2\frac{\nu^{n+1}\tgij}{u^3}uw(F-\sigma)\\
&=\frac{1}{u}\{[(F-\sigma)u]_{ij}-u(F-\sigma)\thh^k_i\thh_{kj}\}-\frac{\thij}{u}w(F-\sigma)\\
&-\{\tgukl[u(F-\sigma)]_ku_l\}\frac{\tgij}{u^2}-2\frac{(F-\sigma)\nu^{n+1}}{u}\thij-2\frac{\tgij}{u^2}(F-\sigma).
\end{aligned}
\ee
\end{proof}

\begin{lemma}\label{Evol3}(Evolution of F) The term $F$ evolves according to the equation
\be\label{Evo14}\begin{aligned}&F_t=uF^{ij}[(F-\sigma)u]_i^j+(F-\sigma)\lll[\sum f_s\kappa^2_s-2\nu^{n+1}F+(\nu^{n+1})^2\sum f_s\rrr]\\
&+w(F-\sigma)\lll(F-\nu^{n+1}\sum f_s\rrr)-[(F-\sigma)u]_iu^i\sum f_s.\end{aligned}\ee
\end{lemma}
\begin{proof}
 We consider $F$ with respect to the mixed tensor $h^j_i.$ By equation (\ref{Evo10}) and (\ref{Evo11}) we have
\be\label{Evo2}
\begin{aligned}
&F_t=F^{ij}(h^j_i)_t=F^{ij}\lll(u\thh^j_i+\nu^{n+1}\delta_{ij}\rrr)_t\\
&=uF^{ij}[(F-\sigma)u]^j_i+u^2(F-\sigma)F^{ij}\thh^k_i\thh^j_k\\
&+uw(F-\sigma)F^{ij}\thh^j_i-[(F-\sigma)u]_iu^i\sum f_s\\
&=uF^{ij}[(F-\sigma)u]_i^j+(F-\sigma)\{\sum f_s\kappa_s^2-2\nu^{n+1}F+(\nu^{n+1})^2\sum f_s\}\\
&+w(F-\sigma)(F-\nu^{n+1}\sum f_s)-[(F-\sigma)u]_iu^i\sum f_s.
\end{aligned}
\ee
\end{proof}

\section{Preserving convexity} \label{Pre}
Let $u$ be an admissible solution of (\ref{Int3}) on the domain $\ol\Omega\times[0,T).$ In this section,
we are going to prove that if the initial surface is convex, then during the evolution,
the graph $\Sigma(t)=(x,u(x,t))$ stays convex for $t\in[0,T)$. For convenient, from now on we always choose $\tau_1, \cdots, \tau_n$ to be orthonormal in hyperbolic metrics, i.e., $g_{ij}=\delta_{ij}$ and $\tg_{ij}=u^2\delta_{ij}.$

\begin{lemma}\label{Prel0}
If the initial surface $\Sigma_0$ is convex, then for any $t\in[0,T),$
the flow surface $\Sigma(t)$ stays convex, what's more, if $f(\Sigma_0)>\sigma,$
then $f(\Sigma(t))>\sigma,\;\;\mbox{$(x,t)\in\Omega\times(0, T)$}$.
\end{lemma}
\begin{proof}
Combining equation (\ref{Foh10}) and Lemma \ref{Evol3} we have
\be\label{Pre0}
\begin{aligned}
&\frac{\partial F}{\partial t}-F^{ij}\na_{ij}F\\
&=\F\lll[\sum f_s\ka_s^2-\nu^{n+1}F+(\nu^{n+1})^2\sum f_s+wF-2\sum f_s\rrr].
\end{aligned}
\ee
Now consider function $\tilde{F}=e^{-\lambda t}\F,$ where $\lambda>0$ to be determined later. By equation (\ref{Pre0}) we know that $\tilde{F}$ satisfies
\be\label{Pre1}
\begin{aligned}
&\frac{\partial\tilde{F}}{\partial t}-F^{ij}\na_{ij}\tilde{F}\\
&=\tilde{F}\lll[\sum f_s\ka_s^2-\nu^{n+1}F+(\nu^{n+1})^2\sum f_s+wF-2\sum f_s-\lambda\rrr].
\end{aligned}
\ee
If $\tilde{F}$ achieves its negative minimum at an interior point $(x_0,t_0)\in\Omega_T=\Omega\times(0,T),$
then at this point we would have
\[0>\tilde{F}\lll[\sum f_s\ka_s^2-\nu^{n+1}F+(\nu^{n+1})^2\sum f_s+wF-2\sum f_s-\lambda\rrr].\]
By choosing $\lambda>\max_{\Omega\times[0,T^*]}\lll|f_s\ka_s^2-\nu^{n+1}F+(\nu^{n+1})^2\sum f_s+wF-2\sum f_s\rrr|$, where $0<t_0<T^*<T,$ leads to a contradiction.

Now under the hypothesis $f(\Sigma_0)>\sigma,$ assume there exists a $t_0\in(0,T)$ such that at $(x_0,t_0)\in\Omega\times(0,T)$ $F=\sigma$. Let $\tilde{F}^\epsilon=e^{-\lambda t}(F-\sigma-\epsilon),$ where $0<\epsilon<\inf_{x\in\ol\Omega}\{f(\Sigma_0(x))-\sigma\}.$ Then $\tilde{F}^\epsilon$ would obtain its negative minimum at an interior point while $\tilde{F}^\epsilon(\Sigma_0)>0$ leads to a contradiction.
\end{proof}

Similarly we have
\begin{corollary}\label{Prec0}
Let $\Sigma(t)=\{(x,u(x,t)),(x,t)\in\Omega\times[0,T)\}$ denote the flow surfaces, $f(\Sigma_0)>\sigma,$ and $u$ satisfies equation (\ref{Int1}),  then there exists a constant C only depends on $u_0,$ such that
\be\label{Pre4}
F-\sigma\leq Ce^{\lambda(T^*) t}\;\;\mbox{$\forall t\in[0,T^*),$ $0<T^*<T.$}
\ee
\end{corollary}
\begin{proof}
We still consider the function $\tilde{F}=e^{-\lambda t}\F$ in $\Omega\times[0,T^*),\;0<T^*<T$ where $\lambda$ chosen in the same way as before, then by Lemma \ref{Prel0} we have
\[\frac{\partial\tilde{F}}{\partial t}-F^{ij}\na_{ij}\tilde{F}<0\;\;\mbox{in $\Omega\times[0,T^*).$}\]
 Now we apply maximum principle and conclude that $\tilde{F}$ achieves its maximum at the parabolic boundary. By Theorem \ref{Shol0} we know that $F\equiv\sigma$ on $\partial\Omega\times(0,T),$
 therefore let $C=\max_{x\in\ol\Omega}F(\Sigma_0(x))-\sigma,$ we get (\ref{Pre4}).
\end{proof}

\begin{remark}\label{Prer}
From Corollary \ref{Prec0}, we can see that for any fixed $0<T^*<T,$ there exists a constant $C$ only depends on initial surface $\Sigma_0$ and $T^*,$ such that for any $0\leq t\leq T,$ we have $F<C.$
\end{remark}

\section{Gradient estimates} \label{Gre}
In this section we shall show that for $t\in(0,T)$ an upward unit normal of a solution tends to a fixed asymptotic angle with our axis $e_{n+1}$ on approach to the boundary. Combining this with following results gives us a global gradient bound for the solution.

The following lemma is similar to Theorem 3.1 of \cite{GS10}.
\begin{lemma}\label{Grel0}
Let $\Sigma(t)=\lll\{(x, u(x,t)): (x,t)\in\Omega_T\rrr\}$ be the flow surfaces with $u(x,t)$ is an admissible solution of equation (\ref{Int3}). Then for $\epsilon>0$ sufficiently small,
\be\label{Gre11}
\frac{\sigma-\nu^{n+1}}{u}<\frac{\sqrt{1-\sigma^2}}{r_1}+\frac{\epsilon(1+\sigma)}{r_1^2}\,on\,\partial\Omega\times(0,T),
\ee
where $r_1$ is the maximal radius of exterior tangent sphere to $\partial\Omega.$ Moreover, when $0<t<T$ we have $\nu^{n+1}\rightarrow\sigma$ on $\partial\Omega$ as $\epsilon\rightarrow 0.$
\end{lemma}
\begin{proof}
We first assume $r_1<\infty.$ Let $\Gamma^\epsilon$ denote the vertical $\epsilon$-lift of boundary $\Gamma,$ for a fixed point $x_0\in\Gamma^{\epsilon},$ let $\bf{e}_1$ be the outward unit normal vector to $\Gamma^{\epsilon}$ at $x_0.$ Let $B_1$ be a ball in $\mathbb{R}^{n+1}$ of radius $R_1$ centered at $a=(x_0+r_1\mathbf{e}_1, R_1\sigma)$ where $R_1$ satisfies $R_1^2=r_1^2+(R_1\sigma-\epsilon)^2.$

Note that $B_1\cap P(\epsilon)=\{x\in\mathbb{R}^{n+1}|x_{n+1}=\epsilon\}$ is an n-ball of radius $r_1,$ which externally tangent to $\Gamma^{\epsilon}.$ By Lemma 3.3 of \cite{LX10}, we know that $B_1\cap\Sigma(t)={\emptyset},$ for any $t\in[0,T)$ hence at $x_0,$ we have
\[\nu^{n+1}>-\frac{u-\sigma R_1}{R_1}.\]
By an easy computation we can get
\[R_1\geq\frac{r_1^2}{\sqrt{(1-\sigma^2)r_1^2}+(1+\sigma)\epsilon}\]
thus equation (\ref{Gre11}) is proved.
If $r_1=\infty,$ then in the above argument one can replace $r_1$ by any $r>0$ and then let $r\rightarrow\infty.$

Applying Theorem {\ref{Shol0}}, similarly we derive
$$-\frac{\sqrt{1-\sigma^2}}{r_2}-\frac{\epsilon(1-\sigma)}{r_2^2}<\frac{\sigma-\nu^{n+1}}{u}\;\;\mbox{on $\partial\Omega\times(0,T)$}$$
where $r_2$ is the radius of the largest interior tangent sphere to $\partial\Omega.$
\end{proof}

\begin{proposition}\label{Grep0}
Let $\Sigma(t)$ be the flow surfaces, where $\Sigma(t)=\{(x,u(x,t)):\,(x,t)\in\Omega_T\}$
and $u(x,t)$ satisfies the AMGCF equation (\ref{Int3}). Then
\be\label{Gre0}
\frac{1}{\nu^{n+1}}\leq\max\lll\{\frac{\max_{\Omega_T}u}{u},\max_{\partial\Omega_T}\frac{1}{\nu^{n+1}}\rrr\},
\ee
where $\Omega_T=\Omega\times[0,T).$
\end{proposition}

\begin{proof}
Let $h=uw$ and suppose that $h$ obtains its maximum at an interior point $(x_0,t_0),$ then at this point
we have
\[\partial_i h=(\delta_{ki}+u_ku_i+uu_{ki})\frac{u_k}{w}=0,\,\,\,for \,\forall\,0\leq i\leq n.\]
By Lemma \ref{Prel0} we know that $\Sigma(t_0)$ is strictly locally convex. According to Theorem \ref{Intt0}, this implies that $\na u=0$ at $(x_0, t_0)$, thus the conclusion
follows immediately.
\end{proof}

Now we can apply equation (\ref{Foh9}) and (\ref{Foh10}) to prove the following theorem.
\begin{theorem}\label{Gret0}
Consider the flow surfaces $\Sigma(t),$ where $\Sigma(t)$ is supposed to be globally a graph:
\[\Sigma(t)=\{(x,u(x,t)):\,(x,t)\in\Omega_T\}\]
and $u(x,t)$ satisfies the AMGCF equation (\ref{Int3}), then we have
\be\label{Gre1}
\frac{\sigma-\nu^{n+1}}{u}\leq\max\lll\{\frac{\sigma-\frac{1}{3}\sigma}{u},\max_{\partial\Omega_T}
\frac{\sigma-\nu^{n+1}}{u}\rrr\}.
\ee
\end{theorem}
\begin{proof}
By equation (\ref{Foh2}), (\ref{Foh9}) and let $\tg_{ij}=u^2\delta_{ij}$
\be\label{Gre2}
\begin{aligned}
&\nabla_{ij}\frac{1}{u}=-\frac{1}{u^2}\tilde{\nabla}_{ij}u+\frac{1}{u^3}\tg^{kl}u_ku_l\tg_{ij}\\
&=-\frac{1}{u^2}\thh_{ij}\nu^{n+1}+\frac{1}{u^3}\tg^{kl}u_ku_l\tg_{ij}\\
&=-\frac{\nu^{n+1}}{u}\lll(h_{ij}-\frac{\nu^{n+1}}{u^2}\tg_{ij}\rrr)+\frac{1}{u^3}\tg^{kl}u_ku_l\tg_{ij}
\end{aligned}
\ee
hence,
\be\label{Gre3}
\begin{aligned}
&F^{ij}\na_{ij}\frac{1}{u}=-\frac{\nu^{n+1}}{u}F+\frac{(\nu^{n+1})^2}{u^3}\sum F^{ij}\tg_{ij}
+\frac{1}{u^3}u_ku^k\sum F^{ij}\tg_{ij}\\
&=-\frac{\nu^{n+1}}{u}F+\frac{(\nu^{n+1})^2}{u}\sum f_k+\frac{1-(\nu^{n+1})^2}{u}\sum f_k\\
&=-\frac{\nu^{n+1}}{u}F+\frac{1}{u}\sum f_k.
\end{aligned}
\ee
Moreover,
\be\label{Gre4}
\na_{ij}\frac{\nu^{n+1}}{u}=\nu^{n+1}\na_{ij}\frac{1}{u}+\frac{1}{u}\tna_{ij}\nu^{n+1}
-\frac{1}{u^2}\tg^{kl}u_k(\nu^{n+1})_l\tg_{ij}.
\ee
We recall the identities in $\mathbb{R}^{n+1}$
\be\label{Gre13}
\lll(\nu^{n+1}\rrr)_i=-\thh_{ij}\tg^{jk}u_k
\ee
\be\label{Gre14}
\tna_{ij}\nu^{n+1}=-\tg^{kl}\lll(\nu^{n+1}\thh_{il}\thh_{kj}+u_l\tna_k\thh_{ij}\rrr).
\ee
By equation (\ref{Gre3}), (\ref{Gre4}) and (\ref{Gre14}) we see that
\be\label{Gre5}
\begin{aligned}
&F^{ij}\na_{ij}\frac{\nu^{n+1}}{u}\\
&=\nu^{n+1}F^{ij}\na_{ij}\frac{1}{u}+\frac{1}{u}F^{ij}\tna_{ij}\nu^{n+1}
-\frac{1}{u^2}\tg^{kl}u_k(\nu^{n+1})_lF^{ij}\tg_{ij}\\
&=-\frac{(\nu^{n+1})^2}{u}F+\frac{\nu^{n+1}}{u}\sum f_k+\frac{1}{u}F^{ij}\lll[-\tg^{kl}
(\nu^{n+1}\thh_{il}\thh_{kj}+u_l\tna_k\thh_{ij})\rrr]\\
&-\frac{1}{u^2}\tg^{kl}u_k(\nu^{n+1})_lF^{ij}\tg_{ij}.\\
\end{aligned}
\ee
As a hypersurface in $\mathbb{R}^{n+1},$ it follows from equation (\ref{Foh3}) that for any $0\leq t<T,$ $\Sigma(t)$ satisfies
\be\label{Gre15}
f\lll(u\tilde{\kappa}_1+\nu^{n+1},\cdots, u\tilde{\kappa}_n+\nu^{n+1}\rrr)=F
\ee
or equivalently,
\be\label{Gre16}
F\lll(\left\{u\tg^{sk}\thh_{kr}+\nu^{n+1}\delta_{sr}\right\}\rrr)=F.
\ee
Differentiating equation (\ref{Gre16}) and using $\tg^{sr}=\frac{\delta_{sr}}{u^2}$ we obtain
\be\label{Gre17}
F_i=\frac{u_i}{u}F-\frac{u_i}{u}\nu^{n+1}\sum f_k+\frac{1}{u}F^{sr}\tna_i\thh_{sr}+\lll(\nu^{n+1}\rrr)_i\sum f_k.
\ee
Combining lemma \ref{Evol1} and equation (\ref{Gre17}) we derive
\be\label{Gre6}
\begin{aligned}
&\lll(\frac{\nu^{n+1}}{u}\rrr)_t=\frac{\nu^{n+1}_t}{u}-\frac{\nu^{n+1}}{u^2}u_t\\
&=\frac{1}{u}\lll\{-\tg^{ij}[\F u]_iu_j\rrr\}-\frac{\nu^{n+1}}{u^2}u_t\\
&=-\tg^{ij}F_iu_j-\frac{\F}{u}\tg^{ij}u_iu_j-\frac{\F}{u}\\
&=-u^i\lll(\frac{u_i}{u}F-\frac{u_i}{u}\nu^{n+1}\sum f_k+\frac{1}{u}F^{st}\tna_i\thh_{st}
+(\nu^{n+1})_i\sum f_k\rrr)\\
&-\frac{\F}{u}\tg^{ij}u_iu_j-\frac{\F}{u}\\
&=-\frac{|\tna u|^2}{u}F+\frac{|\tna u|^2}{u}\nu^{n+1}\sum f_k-\frac{u^i}{u}F^{st}\tna_i\thh_{st}-u^i(\nu^{n+1})_i\sum f_k\\
&-\frac{\F}{u}(|\tna u|^2+1).
\end{aligned}
\ee
Finally we get
\be\label{Gre7}
\begin{aligned}
&\lll(\frac{\partial}{\partial t}-F^{ij}\na_{ij}\rrr)\frac{\nu^{n+1}}{u}\\
&=-\frac{|\tna u|^2}{u}F+\frac{|\tna u|^2}{u}\nu^{n+1}\sum f_k-\frac{u^i}{u}F^{st}\tna_i\thh_{st}\\
&-u^i(\nu^{n+1})_i\sum f_k-\frac{\F}{u}(|\tna u|^2+1)
+\frac{(\nu^{n+1})^2}{u}F-\frac{\nu^{n+1}}{u}\sum f_k\\
&-\frac{1}{u}F^{ij}\lll[-\tg^{kl}(\nu^{n+1}\thh_{il}\thh_{kj}+u_l\tna_k\thh_{ij})\rrr]
+\frac{1}{u^2}\tg^{kl}u_k(\nu^{n+1})_lF^{ij}\tg_{ij}\\
&=-\frac{|\tna u|^2}{u}F+\frac{\nu^{n+1}-(\nu^{n+1})^3}{u}\sum f_k-\frac{\F}{u}(|\tna u|^2+1)\\
&+\frac{(\nu^{n+1})^2}{u}F-\frac{\nu^{n+1}}{u}\sum f_k
+\frac{1}{u}F^{ij}\tg^{kl}\nu^{n+1}\thh_{il}\thh_{kj}\\
&=-\frac{1}{u}F-\frac{\F}{u}(|\tna u|^2+1)+\frac{\nu^{n+1}}{u}\sum f_k\kappa_k^2.\\
\end{aligned}
\ee
By a simple computation we have
\be\label{Gre8}
\lll(\frac{\partial}{\partial t}-F^{ij}\na_{ij}\rrr)\frac{1}{u}=
-\frac{\F}{u\nu^{n+1}}+\frac{\nu^{n+1}}{u}F-\frac{1}{u}\sum f_k.
\ee

Therefore,
\be\label{Gre9}
\begin{aligned}
&\lll(\frac{\partial}{\partial t}-F^{ij}\na_{ij}\rrr)\frac{\sigma-\nu^{n+1}}{u}\\
&=-\frac{\sigma\F}{u\nu^{n+1}}+\frac{\sigma\nu^{n+1}}{u}F-\frac{\sigma}{u}\sum f_k
+\frac{1}{u}F\\
&+\frac{\F}{u}(2-(\nu^{n+1})^2)-\frac{\nu^{n+1}}{u}\sum f_k\kappa^2_k\\
&\leq\frac{1}{u}(F-\sigma\sum f_k)+\frac{\F}{u}\lll(2-(\nu^{n+1})^2-\frac{\sigma}{\nu^{n+1}}\rrr)\\
&+\frac{\nu^{n+1}}{u}\lll(\sigma F-\frac{F^2}{\sum f_k}\rrr)\\
&=\frac{1}{u}(F-\sigma\sum f_k)\lll(1-\frac{\nu^{n+1}F}{\sum f_k}\rrr)+\frac{\F}{u}\lll(2-(\nu^{n+1})^2-\frac{\sigma}{\nu^{n+1}}\rrr)
\end{aligned}
\ee
where we applied inequality $\sum f_k\kappa_k^2\geq \frac{F^2}{\sum f_k}.$
If $\frac{\sigma-\nu^{n+1}}{u}$ achieves its maximum at an interior point $(x_0,t_0),$ then at this point
we have
\be\label{Gre10}
\begin{aligned}
&0\leq\frac{1}{u}(F-\sigma\sum f_k)\lll(1-\frac{\nu^{n+1}F}{\sum f_k}\rrr)+\frac{\F}{u}\lll(2-(\nu^{n+1})^2-\frac{\sigma}{\nu^{n+1}}\rrr)\\
&=\frac{F-\sigma}{u}\lll(2-(\nu^{n+1})^2-\frac{\sigma}{\nu^{n+1}}+\frac{F-\sigma\sum f_k}{\F}
-\frac{\nu^{n+1}F}{\sum f_k}\frac{(F-\sigma\sum f_k)}{\F}\rrr).
\end{aligned}
\ee
when $F\geq\sigma\sum f_k$
\[0\leq\frac{\F}{u}\lll(3-\frac{\sigma}{\nu^{n+1}}\rrr),\]
when $F<\sigma\sum f_k$
\[0\leq\frac{\F}{u}\lll(2-\frac{\sigma}{\nu^{n+1}}\rrr).\]
Thus by Lemma \ref{Prel0} we have when $\nu^{n+1}<\frac{\sigma}{3}$ at $(x_0,t_0)$,
$$\lll(\frac{\partial}{\partial t}-F^{ij}\na_{ij}\rrr)\frac{\sigma-\nu^{n+1}}{u}<0$$ leads to a contradiction.

Therefore we conclude that
\[\frac{\sigma-\nu^{n+1}}{u}\leq\max\lll\{\frac{\sigma-\frac{1}{3}\sigma}{u},\max_{\partial\Omega_T}
\frac{\sigma-\nu^{n+1}}{u}\rrr\}.\]
\end{proof}

Combining  Lemma \ref{Grel0}, Proposition \ref{Grep0} and Theorem \ref{Gret0} gives
\begin{corollary}\label{Grec0}
For any $\epsilon>0$ sufficiently small, any admissible solution $u^{\epsilon}$
of the Dirichlet problem (\ref{Int3}) satisfies the a priori estimates
\be\label{Gre12}
|\na u^{\epsilon}|\leq C \,\,\mbox{in $\ol{\Omega}_T$},
\ee
where $C$ is independent of $\epsilon$ and $T.$
\end{corollary}

\section{$C^2$ boundary estimates} \label{C2b}
In this section, we will establish boundary estimates for second order spatial derivatives of the admissible solutions to the Dirichlet problem (\ref{Int3}). Following the notations in subsection \ref{Fohsecond} we can rewrite equation (\ref{Int3}) as follows:
\be\label{C2b0}
\left\{
\begin{aligned}
&\frac{1}{uw}u_t-F\lll(\frac{1}{w}(\delta_{ij}+u\gamma^{ik}u_{kl}\gamma^{lj})\rrr)=-\sigma\;&\mbox{on $\Omega_T$},\\
&u(x,0)=u_0^\epsilon\;&\mbox{on $\Omega\times\{0\}$},\\
&u(x,t)=\epsilon\;&\mbox{on $\partial\Omega\times[0,T)$}.
\end{aligned}
\right.
\ee
And from now on we denote
\be\label{C2b1}
G(D^2u,Du,u,u_t)=\frac{1}{uw}u_t-F.
\ee

\begin{theorem}\label{C2bt0}
Suppose $f$ satisfies equation (\ref{Int5})-(\ref{Int11}). If $\epsilon$ is sufficiently small,
\be\label{C2b22}
u|D^2u|\leq C\;\mbox{on $\partial\Omega\times[0,T)$},
\ee
where $C$ is independent of $\epsilon.$
\end{theorem}
\begin{remark}\label{C2br0}
The following proof shows that $C$ does not depend on $\epsilon,$ but depends on $T$. In section \ref{c2g} we will show that in fact $C$ is also independent of $T.$
\end{remark}
Note that
\be\label{C2b2}
G^{kl}:=\frac{\partial G}{\partial u_{kl}}=-\frac{u}{w}F^{ij}\gamma^{ik}\gamma^{lj},
\ee
\be\label{C2b3}
G^{kl}u_{kl}=-F+\frac{1}{w}\sum F^{ii},
\ee
\be\label{C2b4}
\begin{aligned}
&G_u:=\frac{\partial G}{\partial u}=-\frac{1}{wu^2}u_t-\frac{1}{w}F^{ij}\gamma^{ik}u_{kl}\gamma^{lj}\\
&=-\frac{\F}{u}-F^{ij}\lll(\frac{a_{ij}}{u}-\frac{1}{uw}\delta_{ij}\rrr)\\
&=-\frac{2F}{u}+\frac{\sigma}{u}+\frac{1}{wu}\sum F^{ii},
\end{aligned}
\ee
\be\label{C2b5}
G^t:=\frac{\partial G}{\partial u_t}=\frac{1}{uw},
\ee
\be\label{C2b6}
\begin{aligned}
&G^s:=\frac{\partial G}{\partial u_s}\\
&=-\frac{u_tu_s}{uw^3}+\frac{u_s}{w^2}F+\frac{2}{w}F^{ij}a_{ik}\lll(\frac{wu_k\gamma^{sj}+u_j\gamma^{ks}}{1+w}\rrr)
-\frac{2}{w^2}F^{ij}u_i\gamma^{sj}\\
&=-\frac{\F}{w^2}u_s+\frac{u_s}{w^2}F+\frac{2}{w}F^{ij}a_{ik}\lll(\frac{wu_k\gamma^{sj}+u_j\gamma^{ks}}{1+w}\rrr)
-\frac{2}{w^2}F^{ij}u_i\gamma^{sj}\\
&=\frac{u_s}{w^2}\sigma+\frac{2}{w}F^{ij}a_{ik}\lll(\frac{wu_k\gamma^{sj}+u_j\gamma^{ks}}{1+w}\rrr)
-\frac{2}{w^2}F^{ij}u_i\gamma^{sj}.\\
\end{aligned}
\ee
Thus
\be\label{C2b7}
G^su_s=\frac{w^2-1}{w^2}\sigma+\frac{2}{w^2}F^{ij}a_{ik}u_ku_j-\frac{2}{w^3}F^{ij}u_iu_j.
\ee
And similar to equation (5.4) in \cite{GS08} we have
\be\label{C2b24}
\sum |G^s|\leq C(\sum F^{ii}+F).
\ee

Next, we consider the partial linearized operator of $G$ at $u$:
\[L=G^t\partial_t+G^{kl}\partial_k\partial_l+G^s\partial_s.\]
By equation (\ref{C2b3}),(\ref{C2b5}) and (\ref{C2b7}) we get
\be\label{C2b8}
\begin{aligned}
&Lu=\frac{1}{uw}u_t-F+\frac{1}{w}\sum F^{ii}+\frac{w^2-1}{w^2}\sigma
+\frac{2}{w^2}F^{ij}a_{ik}u_ku_j-\frac{2}{w^3}F^{ij}u_iu_j\\
&=-\frac{1}{w^2}\sigma+\frac{1}{w}\sum F^{ii}+\frac{2}{w^2}F^{ij}a_{ik}u_ku_j-\frac{2}{w^3}F^{ij}u_iu_j,\\
\end{aligned}
\ee
hence
\be\label{C2b9}
\begin{aligned}
&L\lll(\frac{1}{u}\rrr)=-\frac{1}{u^2}Lu+\frac{2}{u^3}G^{kl}u_ku_l\\
&=\frac{1}{u^2w^2}\sigma-\frac{1}{u^2w}\sum F^{ii}-\frac{2}{u^2w^2}F^{ij}a_{ik}u_ku_j\\
&+\frac{2}{w^3u^2}F^{ij}u_iu_j-\frac{2}{u^2w}F^{ij}\gamma^{is}u_s\gamma^{rj}u_r\\
&=\frac{1}{w^2u^2}\sigma-\frac{1}{wu^2}\sum F^{ii}-\frac{2}{w^2u^2}F^{ij}a_{ik}u_ku_j.\\
\end{aligned}
\ee

\begin{lemma}\label{C2bl0}
Suppose that $f$ satisfies equation (\ref{Int5}), (\ref{Int6}), (\ref{Int9}) and (\ref{Int10}). Then
\be\label{C2b10}
L\lll(1-\frac{\epsilon}{u}\rrr)\geq \frac{\epsilon(1-\sigma)}{wu^2}\sum F^{ii}\;\mbox{in $\Omega_T$.}
\ee
\end{lemma}
\begin{proof}
Since $\{F^{ij}\}$ and $\{a_{ij}\}$ are both positive definite and can be diagonalized simultaneously, we see that
\be\label{C2b11}
F^{ij}a_{ik}\xi_k\xi_j\geq 0,\,\,\forall \xi\in\mathbb{R}^n.
\ee
Combining with equation (\ref{C2b9})
\be\label{C2b12}
\begin{aligned}
&L\lll(1-\frac{\epsilon}{u}\rrr)=-\epsilon L\lll(\frac{1}{u}\rrr)\\
&=\frac{-\epsilon}{u^2w^2}\sigma+\frac{\epsilon}{u^2w}\sum F^{ii}+\frac{2\epsilon}{w^2u^2}F^{ij}a_{ik}u_ku_j\\
&\geq\frac{\epsilon\lll(1-\frac{\sigma}{w}\rrr)}{wu^2}\sum F^{ii}\geq\frac{\epsilon(1-\sigma)}{u^2w}\sum F^{ii}.
\end{aligned}
\ee
\end{proof}

Now we denote $\mathfrak{L}=G^t\partial t+G^{kl}\partial_k\partial_l+G^s\partial_s+G_u,$ similar to \cite{CNS84} we have
\begin{lemma}\label{C2bl3}
Suppose that $f$ satisfies equation (\ref{Int5}), (\ref{Int6}), (\ref{Int9}) and (\ref{Int10}). Then
\be\label{C2b23}\mathfrak{L}(x_iu_j-x_ju_i)=0,\,\,\mathfrak{L}(u_i)=0,\,\,1\leq i, j\leq n.\ee
\end{lemma}

\begin{proof}[Proof of Theorem~\ref{C2bt0}] Consider an arbitrary point on $\partial\Omega,$ which we may assume to be the origin of $\mathbb{R}^n$ and choose the coordinates so that the positive $x_n$ axis is the interior normal to $\partial\Omega$ at the origin. There exists a uniform constant $r>0$ such that $\partial\Omega\cap B_r(0)$ can be represented as a graph
\[x_n=\rho(x')=\frac{1}{2}\sum_{\alpha,\beta <n}B_{\alpha\beta}x_{\alpha}x_{\beta}+O(|x'|^3),\;x'=(x_1,\cdots,x_{n-1}).\]
Since $u\equiv\epsilon,\;\mbox{on $\partial\Omega\times[0,T),$}$ i.e., $u(x',\rho(x'))\equiv\epsilon$ for $\forall t\in[0,T),$
at the origin we have
\[u_{\alpha}+u_nB_{\alpha\beta}x_{\beta}=0,\,\,u_{\alpha\beta}+u_n\rho_{\alpha\beta}=0,\,\,\forall t\in[0,T)\;\mbox{and $\alpha, \beta<n$}.\]
Consequently,
\be\label{C2b25}
|u_{\alpha\beta}(0,t)|\leq C|Du(0,t)|,\;\;\forall t\in[0,T)\;\mbox{and $\alpha,\beta<n,$}
\ee
where $C$ depends only on the maximal (Euclidean principal) curvature of $\partial\Omega.$
Following \cite{CNS84} let $T_{\alpha}=\partial_{\alpha}+\sum_{\beta<n}B_{\alpha\beta}(x_{\beta}\partial_{n}-x_n\partial_{\beta}),$
then for fixed $\alpha<n,$ we have
\be\label{C2b13}
|T_\alpha u|\leq C_1|x|^2,\;\mbox{on $\{\partial\Omega\cap B_{\epsilon}(0)\}\times[0, T),$}
\ee
\be\label{C2b14}
|T_{\alpha}u|\leq C_1,\;\mbox{in $\{\Omega\cap B_{\epsilon}(0)\}\times[0,T),$}
\ee
where $C_1$ is independent of $\epsilon$ and $T.$
Moreover by Lemma \ref{C2bl3}
\be\label{C2b15}
\mathfrak{L}T_{\alpha}u=0.
\ee
Therefore
\be\label{C2b16}
\begin{aligned}
\lll|L(T_{\alpha}u)\rrr|&=\lll|\mathfrak{L}(T_{\alpha}u)-G_uT_\alpha u\rrr|\\
&=|G_uT_\alpha u|\leq C_1|G_u|\\&\leq\frac{C_2}{u}(\sum F^{ii}+F)\\&\leq \frac{C_2}{u}\sum F^{ii}\;\;\mbox{in $\{\Omega\cap B_{\epsilon}(0)\}\times[0, T).$}
\end{aligned}
\ee
Note that the last inequality comes from equation (\ref{Int13}), Corollary \ref{Prec0} and Remark \ref{Prer}. Hence $C_2$ is some constant only depending on $T.$
By equation (\ref{C2b2}), (\ref{C2b24}) and Lemma 2.1 in \cite{GS08}
\be\label{C2b17}
\begin{aligned}
&|L(|x|^2)|=\lll|G^{kl}\partial_k\partial_l(|x|^2)+G^s\partial_s(|x|^2)\rrr|\\
&=|2\sum G^{kk}+2\sum x_s G^s|\\
&\leq C_3(u\sum F^{ii}+\epsilon|G_s|)\leq C_3u\sum F^{ii}\;\;\mbox{in $\{\Omega\cap B_{\epsilon}(0)\}\times[0,T),$}\\
\end{aligned}
\ee
for the same reason as before we know that $C_3$ only depends on $T$ as well.

Now consider function $$\Phi=A\lll(1-\frac{\epsilon}{u}\rrr)+B|x|^2\pm T_\alpha u.$$
First choose $B\geq\frac{C_1}{\epsilon^2},$
then we have $\Phi\geq 0$ on $\{\partial(\Omega\cap B_\epsilon)\}\times[0,T).$

Next consider $\Phi$ on $(\Omega\cap B_{\delta}(0))\times\{0\},$ where $\delta>\epsilon>0$ is small enough. By using Taylor's theorem we have
\[\begin{aligned}
&\Phi=A\lll(1-\frac{\epsilon}{u_0}\rrr)+B|x|^2\pm T_{\alpha}u_0\\
&\geq A\lll(1-\frac{\epsilon}{\epsilon+a_1x_n}\rrr)+B|x|^2-b_1x_n-b_2|x|^2\\
&\geq\lll(\frac{Aa_1}{1+a_1}-b_1\rrr)x_n+\lll(B-b_2\rrr)|x|^2,
\end{aligned}\]
where $u_0\geq\epsilon+ a_1x_n,$ $|T_{\alpha}u_0|\leq b_1x_n+b_2|x|^2$ in $\Omega\cap B_{\delta}(0)$ and $a_1, b_1, b_2>0.$
(The reason of the existence of $a_1$ can be found in section 3 of \cite{LX10} while the existence of $b_i$, $i=1, 2$ is trivial. )
Hence we conclude that when $A\geq\frac{b_1(1+a_1)}{a_1}$ and $B\geq\max\{\frac{C_1}{\epsilon^2}, b_2\},$ $\Phi\geq 0$ on
$\left\{\partial(\Omega\cap B_{\epsilon}(0))\times[0,T)\right\}\cap\left\{(\Omega\cap B_{\epsilon}(0))\times\{0\}\right\}.$

Moreover, by (\ref{Int13}),(\ref{C2b16}),(\ref{C2b16}) and Lemma \ref{C2bl0}
\be\label{C2b18}
\begin{aligned}
&L(\Phi)=AL\lll(1-\frac{\epsilon}{u}\rrr)+BL(|x|^2)\pm L(T_\alpha u)\\
&\geq\frac{A\epsilon(1-\sigma)}{u^2w}-C_3Bu-\frac{C_2}{u}.
\end{aligned}
\ee
Choosing $A\gg\frac{C_1C_3+C_2}{1-\sigma}$ such that $L\Phi\geq 0$ in $\{\Omega\cap B_{\epsilon}\}\times[0, T),$
which implies that $\Phi\geq 0$ in $\{\Omega\cap B_\epsilon\}\times[0,T).$ Since $\Phi(0,t)=0,$ we have $\Phi_n(0,t)\geq 0,$ for any fixed $t\in[0,T).$
Thus
\be\label{C2b19}
A\lll(\frac{\epsilon}{u^2}u_n\rrr)\pm (T_{\alpha}u)_n\geq 0
\ee
which implies, for any fixed $t\in[0,T),$
\be\label{C2b20}
|u_{\alpha n}(0,t)|\leq\frac{Au_n(0,t)}{u(0,t)}
\ee

Since when $t=0,$ $u_{nn}(0,0)$ is given we only care about the case when $t>0.$ By Theorem \ref{Shol0}, we know that $F\equiv\sigma,\;\mbox{on $\partial\Omega\times(0,T)$}.$ Therefore we can establish $|u_{nn}(0,t)|$, $\forall t\in(0,T)$ in the same way as \cite{GSZ09}. For completeness we include the argument here.

For a fixed $t\in(0,T),$ we may assume $\lll(u_{\alpha\beta}(0,t)\rrr)_{1\leq\alpha,\beta<n}$ to be diagonal. Then at the point $(0,t)$
$$A[u]=\frac{1}{w}
\lll[\begin{array}{cccc}
1+uu_{11}&0&\cdots&\frac{uu_{1n}}{w}\\
0&1+uu_{22}&\cdots&\frac{uu_{2n}}{w}\\
\vdots&\vdots&\ddots&\vdots\\
\frac{uu_{n1}}{w}&\frac{uu_{n2}}{w}&\cdots&1+\frac{uu_{nn}}{w^2}\\
\end{array}\rrr] $$

By lemma 1.2 in \cite{CNS85}, if $\epsilon u_{nn}(0,t)$ is very large, the eigenvalues $\lambda_1,\cdots,\lambda_n$
of $A[u]$ are given by
\be\label{C2b21}
\begin{aligned}
&\lambda_\alpha=\frac{1}{w}(1+\epsilon u_{\alpha\alpha}(0,t))+o(1),\,\alpha<n\\
&\lambda_n=\frac{\epsilon u_{nn}(0,t)}{w^3}\lll(1+O\lll(\frac{1}{\epsilon u_{nn}(0,t)}\rrr)\rrr).
\end{aligned}
\ee
If $\epsilon u_{nn}\geq R$ where $R$ is a uniform constant, then by (\ref{Int10}), (\ref{Int11}) and Lemma \ref{Grel0} we have
\[\sigma=\frac{1}{w}F(wA[u])(0,t)\geq(\sigma-C\epsilon)\lll(1+\frac{\epsilon_0}{2}\rrr)>\sigma\]
which is a contradiction. Therefore
\[|u_{nn}(0,t)|\leq\frac{R}{\epsilon}\]
and the proof is completed.
\end{proof}

\section{$C^2$ global estimates}\label{c2g}
Let $\Sigma(t)=\{(x,u(x,t))\mid x\in\Omega, t\in[0,T)\}$ be the flow surfaces in $\mathbb{H}^{n+1}$ where $u(x,t)$ satisfies $u_t=uw\F.$
For a fixed point $\mathbf{x}_0\in\Sigma(t_0),\;0<t_0<T$ we choose a local orthonormal frame $\tau_1,\cdots,\tau_n$ around $\mathbf{x}_0$ such that
$h_{ij}(\mathbf{x}_0)=\kappa_i\delta_{ij},$ where $\kappa_1,\cdots,\kappa_n$ are the hyperbolic principal curvature of $\Sigma(t_0)$ at $\mathbf{x}_0.$ The calculations below are done at $\mathbf{x}_0.$ In this section, for convenience we shall write $v_{ij}=\na_{ij}v,\;h_{ijk}=\na_{k}h_{ij},\;h_{ijkl}=\na_{lk}h_{ij},$ etc.

\begin{theorem}\label{c2gt0}
Let $\Sigma(t)=\{(x,u(x,t))\mid x\in\Omega, t\in[0,T)\}$ be the flow surfaces in $\mathbb{H}^{n+1}$ where $u(x,t)$ satisfies AMGCF equation (\ref{Int3}) and
\[\nu^{n+1}\geq2a>0\;\;\mbox{on $\Sigma(t),\;\forall t\in[0,T)$}.\]
For $\mathbf{x}\in\Sigma(t),$ let $\ka_{\max}(\mathbf{x})$ be the largest principal curvature of $\Sigma(t)$ at $\mathbf{x}.$ Then
\be\label{c2g13}
\max_{\ol{\Omega}_T}\frac{\ka_{\max}}{\nu^{n+1}-a}
\leq\max\left\{\frac{4}{a^3},\max_{\partial\Omega_T}\frac{\ka_{\max}}{\nu^{n+1}-a}\right\},
\ee
where $\Omega_T=\Omega\times[0,T).$
\end{theorem}

Since the proof of this Theorem is very complicated, we shall divide it into several parts.

To begin with, we denote
\be\label{c2g14}
M_0=\max_{\ol{\Omega}_T}\frac{\ka_{\max}(x)}{\nu^{n+1}-a}.
\ee
Without loss of generality we may assume $M_0>0$ is attained at an interior point $\mathbf{x}_0\in\Sigma(t_0),\;t_0\in(0,T).$ We may also assume $\kappa_1=\kappa_{\max}(\mathbf{x}_0).$
Thus we say at $\mathbf{x}_0,$ $\frac{h_{11}}{\nu^{n+1}-a}$ achieves its local maximum. Hence,
\be\label{c2g15}
\frac{h_{11i}}{h_{11}}-\frac{\na_i\nu^{n+1}}{\nu^{n+1}-a}=0,
\ee
\be\label{c2g16}
\frac{h_{11ii}}{h_{11}}-\frac{\na_{ii}\nu^{n+1}}{\nu^{n+1}-a}\leq 0.
\ee

\begin{lemma}\label{c2gl0}
At $\mathbf{x}_0\in\Sigma(t_0),$ $t_0\in(0,T),$
\be\label{c2g17}
\begin{aligned}
\frac{\partial}{\partial t} h_{11}&=\na_{11}F-\F\kappa_1^2+\kappa_1\nu^{n+1}\F\\&-\frac{\kappa_1}{\nu^{n+1}}\F+\F(\nu^{n+1})^2-2\F.
\end{aligned}
\ee
\end{lemma}
\begin{proof}
By Lemma \ref{Evol2} equation (\ref{Evo13}) and $\tg_{ij}=u^2\delta_{ij}$ we have,
\be\label{c2g0}
\begin{aligned}
&\frac{\partial}{\partial t}h_{11}=\frac{1}{u}\{\tna_{11}[\F u]-u\F\thh^k_1\thh_{k1}\}-\frac{\thh_{11}}{u}w\F\\
&-[u \F]_ku^k-\frac{2 \F \nu^{n+1}}{u}\thh_{11}-2\F.
\end{aligned}
\ee
Recall equation (\ref{Foh10}) we get
\begin{align*}
&\tna_{11}[\F u]=\na_{11}[\F u]-\frac{1}{u}\lll\{2u_1[\F u]_1-u^k[\F u]_k u^2\rrr\}\\
&=u\na_{11}F+\F\na_{11}u+2F_1u_1-\frac{2}{u}\lll\{uu_1F_1+u_1^2\F\rrr\}+uu^k[\F u]_k\\
&=u\na_{11}F+\F\na_{11}u-\frac{2u_1^2\F}{u}+u^k[\F u]_k u,
\end{align*}
inserting this into (\ref{c2g0})
\be\label{c2g1}
\begin{aligned}
&\frac{\partial}{\partial t}h_{11}=\frac{1}{u}\lll\{u\na_{11}F+\F\na_{11}u-\frac{2u_1^2\F}{u}+u^k[\F u]_ku\rrr\}\\
&-\F\thh^k_1\thh_{k1}-\frac{\thh_{11}}{u}w\F-[u\F]_ku^k\\&-\frac{2\F\nu^{n+1}}{u}\thh_{11}-2\F\\
&=\na_{11}F+\frac{\F}{u}\na_{11}u-\frac{2u_1^2\F}{u^2}-\F\thh^k_1\thh_{k1}\\
&-\frac{\thh_{11}}{u}w\F-\frac{2\F\nu^{n+1}}{u}\thh_{11}-2\F.
\end{aligned}
\ee
Note that,
\[\na_{11}u=\tna_{11}u+\frac{2u_1^2}{u}-u|\tna u|^2,\]
\[\frac{\thh_{11}}{u}=h_{11}-\nu^{n+1},\]
\[\thh^k_1\thh_{k1}=\frac{1}{u^2}\thh_{1k}^2=\frac{1}{u^2}(uh_{1k}-u\nu^{n+1}\delta_{1k})^2=(h_{11}-\nu^{n+1})^2.\]
So we have,
\be\label{c2g2}
\begin{aligned}
&\frac{\partial}{\partial t} h_{11}=\na_{11}F+\frac{\F}{u}\lll(\thh_{11}\nu^{n+1}+\frac{2u_1^2}{u}-u|\tna u|^2\rrr)\\
&-\frac{2u_1^2}{u^2}\F-\F(h_{11}-\nu^{n+1})^2-(h_{11}-\nu^{n+1})w\F\\
&-2\F\nu^{n+1}(h_{11}-\nu^{n+1})-2\F\\
&=\na_{11}F+\F\nu^{n+1}(h_{11}-\nu^{n+1})-\F(1-(\nu^{n+1})^2)\\
&-\F(h_{11}^2-2h_{11}\nu^{n+1}+(\nu^{n+1})^2)-(h_{11}-\nu^{n+1})w\F \\
&-2\F\nu^{n+1}(h_{11}-\nu^{n+1})-2\F\\
&=\na_{11}F-\F-\F\kappa_1^2+2\kappa_1\nu^{n+1}\F\\
&-\frac{\kappa_1}{\nu^{n+1}}\F+\F-\F\nu^{n+1}(\kappa_1-\nu^{n+1})-2\F\\
&=\na_{11}F-\F\kappa_1^2+\kappa_1\nu^{n+1}\F-\frac{\kappa_1}{\nu^{n+1}}\F\\
&+\F(\nu^{n+1})^2-2\F.
\end{aligned}
\ee
\end{proof}

\begin{proof}[proof of Theorem \ref{c2gt0}] Now we denote $\varphi=\frac{h_{11}}{\nu^{n+1}-a},$ where $\nu^{n+1}\geq 2a>0$ on $\ol{\Omega}_T$ Then
at $\mathbf{x}_0\in\Sigma(t_0)$, we have
\be\label{c2g3}
\na_i\varphi=\frac{h_{11i}}{\nu^{n+1}-a}-\frac{h_{11}\nu^{n+1}_i}{(\nu^{n+1}-a)^2}=0
\ee
\be\label{c2g4}
\na_{ii}\varphi=\frac{h_{11ii}}{\nu^{n+1}-a}-\frac{h_{11}\na_{ii}\nu^{n+1}}{(\nu^{n+1}-a)^2}\leq 0.
\ee
Using Lemma \ref{c2gl0} and equation (\ref{Evo10}) in Lemma \ref{Evol1} we get
\be\label{c2g5}
\begin{aligned}
&\frac{\partial}{\partial t}\varphi=\frac{\dot{h}_{11}}{\nu^{n+1}-a}-\frac{h_{11}\dot{\nu}^{n+1}}{(\nu^{n+1}-a)^2}\\
&=\frac{1}{\nu^{n+1}-a}\left\{F^{ii}h_{ii11}+F^{ij,rs}h_{ij1}h_{rs1}-\F\kappa_1^2\right.\\
&\left.+\kappa_1\nu^{n+1}\F-\frac{\kappa_1}{\nu^{n+1}}\F+\F(\nu^{n+1})^2-2\F\right\}\\
&+\frac{h_{11}}{(\nu^{n+1}-a)^2}u^k[\F u]_k.
\end{aligned}
\ee
By equation (\ref{Foh10}) and (\ref{Gre14})
\begin{align*}
&\na_{ii}\nu^{n+1}=\tna_{ii}\nu^{n+1}+\frac{1}{u}\lll(2u_i\nu^{n+1}_i-u^k\nu^{n+1}_k\tg_{ii}\rrr)\\
&=-\tg^{kl}\lll(\nu^{n+1}\thh_{il}\thh_{ki}+u_l\tna_k\thh_{ii}\rrr)+\frac{2}{u}u_i\nu^{n+1}_i-uu^k\nu^{n+1}_k\delta_{ii},
\end{align*}
we obtain
\begin{align*}
&F^{ii}\na_{ii}\nu^{n+1}=-\nu^{n+1}F^{ii}\tg^{kl}\thh_{il}\thh_{ki}-F^{ii}u^k\tna_k\thh_{ii}\\
&+\frac{2}{u}F^{ii}u_i\nu^{n+1}_i-uu^k\nu^{n+1}_k\sum f_i\\
&=-\nu^{n+1}\lll(\sum f_i\kappa^2_i-2\nu^{n+1}F+(\nu^{n+1})^2\sum f_i\rrr)-F^{ii}u^k\tna_k\thh_{ii}\\
&+\frac{2}{u}F^{ii}u_i\nu^{n+1}_i-uu^k\nu^{n+1}_k\sum f_i.
\end{align*}
What's more, by the Codazzi and Gauss equations we have
\[h_{ii11}-h_{11ii}=(\kappa_i\kappa_1-1)(\kappa_i-\kappa_1)=\ka^2_i\ka_1-\ka_i\ka_1^2-\ka_i+\ka_1,\]
multiplying by $F^{ii}$ and sum over $i,$
\be\label{c2g22}
\sum F^{ii}(h_{ii11}-h_{11ii})=\ka_1\sum f_i\ka_i^2-\ka_1^2F-F+\ka_1\sum f_i.
\ee
Finally we get
\be\label{c2g6}
\begin{aligned}
&\frac{\partial}{\partial t}\varphi-F^{ii}\na_{ii}\varphi=\frac{1}{\nu^{n+1}-a}\left\{\ka_1\sum f_i\ka^2_i-\ka_1^2F-F\right.\\
&+\ka_1\sum f_i+F^{ij,rs}h_{ij1}h_{rs1}-\F\ka_1^2+\ka_1\nu^{n+1}\F\\
&\left.-\frac{\ka_1}{\nu^{n+1}}\F+\F(\nu^{n+1})^2-2\F\right\}\\
&+\frac{\ka_1}{(\nu^{n+1}-a)^2}\lll\{uF_ku^k+|\tna u|^2\F+F^{ii}\na_{ii}\nu^{n+1}\rrr\}\\
&=\frac{1}{\nu^{n+1}-a}\left\{\ka_1\sum f_i\ka^2_i-\ka_1^2F-F+\ka_1\sum f_i\right.\\
&+F^{ij,rs}h_{ij1}h_{rs1}-\F\ka_1^2+\ka_1\nu^{n+1}\F\\
&\left.-\frac{\ka_1}{\nu^{n+1}}\F+\F(\nu^{n+1})^2-2\F\right\}\\
&+\frac{\ka_1}{(\nu^{n+1}-a)^2}\left\{u\lll(\frac{|\tna u|^2}{u}F-\frac{|\tna u|^2}{u}\nu^{n+1}\sum f_i\rrr)\right.\\
&+|\tna u|^2\F-\nu^{n+1}\lll[\sum f_i\ka_i^2-2\nu^{n+1}F+(\nu^{n+1})^2\sum f_i\rrr]\\
&\left.+\frac{2}{u}\sum F^{ii}u_i\nu_{i}^{n+1}\right\},
\end{aligned}
\ee
where we have used equation (\ref{Gre17}).
Hence at $\mathbf{x}_0\in\Sigma(t_0)$ we have
\be\label{c2g7}
\begin{aligned}
&0\leq\ka_1f_i\ka_i^2-\ka_1^2F-F+\ka_1\sum f_i+F^{ij,rs}h_{ij1}h_{rs1}\\
&-\F\ka_1^2+\ka_1\nu^{n+1}\F-\frac{\ka_1}{\nu^{n+1}}\F\\
&+\F(\nu^{n+1})^2-2\F+\frac{\ka_1}{\nu^{n+1}-a}\left\{|\tna u|^2F\right.\\
&-|\tna u|^2\nu^{n+1}\sum f_i+|\tna u|^2\F-\nu^{n+1}\left[\sum f_i\ka_i^2-2\nu^{n+1}F\right.\\
&\left.\left.+(\nu^{n+1})^2\sum f_i\right]+\frac{2}{u}\sum F^{ii}u_i\nu_i^{n+1}\right\},
\end{aligned}
\ee
which implies
\be\label{c2g8}
\begin{aligned}
&0\leq\lll(-1-\ka_1^2+\ka_1\frac{1+(\nu^{n+1})^2}{\nu^{n+1}-a}\rrr)F+F^{ij,rs}h_{ij1}h_{rs1}\\
&+\lll(\ka_1-\frac{\ka_1\nu^{n+1}}{\nu^{n+1}-a}\rrr)(\sum f_i+\sum f_i\ka_i^2)
+\frac{2\ka_1}{\nu^{n+1}-a}\sum f_i\frac{u_i^2}{u^2}(\nu^{n+1}-\ka_i)\\
&+\F\ka_1\lll(-\ka_1+\nu^{n+1}-\frac{1}{\nu^{n+1}}+\frac{1-(\nu^{n+1})^2}{\nu^{n+1}-a}\rrr)-\F.\\
\end{aligned}
\ee
Next we use an inequality due to Andrews \cite{A94} and Gerhardt \cite{G96} which states
\be\label{c2g18}
-F^{ij,kl}h_{ij1}h_{kl1}\geq 2\sum_{i\geq2}\frac{f_i-f_1}{\kappa_1-\kappa_i}h^2_{i11}.
\ee
Meanwhile at $\mathbf{x}_0\in\Sigma(t_0),$ we obtain from equation (\ref{Gre13}) and (\ref{c2g3})
\be\label{c2g19}
h_{11i}=\frac{\kappa_1}{\nu^{n+1}-a}\frac{u_i}{u}(\nu^{n+1}-\kappa_i).
\ee
Inserting into (\ref{c2g18}) we derive
\be\label{c2g20}
F^{ij,rs}h_{ij1}h_{rs1}\leq 2\left(\frac{\ka_1}{\nu^{n+1}-a}\right)^2
\sum_{i\geq 2}\frac{f_1-f_i}{\ka_1-\ka_i}\frac{u_i^2}{u^2}(k_i-\nu^{n+1})^2.
\ee
Moreover we may write
\be\label{c2g21}
\sum f_i+\sum f_i\ka_i^2=(1-(\nu^{n+1})^2)\sum f_i+\sum(\ka_i-\nu^{n+1})^2f_i+2F\nu^{n+1}.
\ee
Combining equation (\ref{c2g8}), (\ref{c2g20}) and (\ref{c2g21}) gives
\be\label{c2g9}
\begin{aligned}
&0\leq\lll(-1-\ka_1^2+\frac{1+(\nu^{n+1})^2}{\nu^{n+1}-a}\ka_1\rrr)F+\frac{2\ka_1}{\nu^{n+1}-a}\sum f_i\frac{u_i^2}{u^2}(\nu^{n+1}-\ka_i)\\
&-\frac{a\ka_1}{\nu^{n+1}-a}\lll((1-(\nu^{n+1})^2)\sum f_i+\sum(\ka_i-\nu^{n+1})^2f_i+2F\nu^{n+1}\rrr)\\
&-2\lll(\frac{\ka_1}{\nu^{n+1}-a}\rrr)^2\sum_{i\geq 2}\frac{f_i-f_1}{\ka_1-\ka_i}\frac{u_i^2}{u^2}(\ka_i-\nu^{n+1})^2\\
&+\F\ka_1\lll(-\ka_1+\nu^{n+1}-\frac{1}{\nu^{n+1}}+\frac{1-(\nu^{n+1})^2}{\nu^{n+1}-a}\rrr)-\F.
\end{aligned}
\ee
Note that (assuming $\ka_1\geq\frac{2}{a}$) all terms on the right hand side are negative except possibly the ones in the sum involving $(\nu^{n+1}-\ka_i)$ and only if $\ka_i<\nu^{n+1}.$

Therefore define
\[I=\left\{i:\ka_i-\nu^{n+1}\leq-\theta\ka_1\right\},\]
\[J=\left\{i:-\theta\ka_1<\ka_i-\nu^{n+1}<0,\,f_i<\theta^{-1}f_1\right\},\]
\[L=\left\{i:-\theta\ka_1<\ka_i-\nu^{n+1}<0,\,f_i\geq\theta^{-1}f_1\right\},\]
where $\theta\in(0,1)$ is to be chosen later. We get
\be\label{c2g10}
\begin{aligned}
&\frac{-1}{\nu^{n+1}-a}\sum_{i\in I}(\ka_i-\nu^{n+1})^2f_i\\
&\leq\frac{\theta\ka_1}{\nu^{n+1}-a}\sum_{i\in I}(\ka_i-\nu^{n+1})f_i\\
&\leq\frac{\theta\ka_1}{\nu^{n+1}-a}\sum_{i\in I}(\ka_i-\nu^{n+1})f_i\frac{u_i^2}{u^2},\\
\end{aligned}
\ee
\be\label{c2g11}
\sum_{i\in J}(\nu^{n+1}-\ka_i)f_i\frac{u_i^2}{u^2}\leq \ka_1f_1.
\ee
Finally
\be\label{c2g12}
\begin{aligned}
&\frac{-2\ka_1^2}{(\nu^{n+1}-a)^2}\sum_{i\in L}\frac{f_i-f_1}{\ka_1-\ka_i}\frac{u_i^2}{u^2}(\ka_i-\nu^{n+1})^2\\
&\leq\frac{-2\ka_1^2}{(\nu^{n+1}-a)^2}\sum_{i\in L}\frac{(1-\theta)f_i}{(1+\theta)\ka_1}(\ka_i-\nu^{n+1})^2\frac{u_i^2}{u^2}\\
&=\frac{2\ka_1}{\nu^{n+1}-a}\sum_{i\in L}f_i\frac{u_i^2}{u^2}(\ka_i-\nu^{n+1})\\
&+\frac{4\theta}{1+\theta}\frac{\ka_1}{(\nu^{n+1}-a)^2}\sum_{i\in L}(\ka_i-\nu^{n+1})^2f_i\frac{u_i^2}{u^2}\\
&-\frac{2\ka_1}{(\nu^{n+1}-a)^2}\sum_{i\in L}f_i\frac{u_i^2}{u^2}(\ka_i^2-(\nu^{n+1}+a)\ka_i+a\nu^{n+1})\\
&\leq\frac{2\ka_1}{\nu^{n+1}-a}\sum_{i\in L}f_i\frac{u_i^2}{u^2}(\ka_i-\nu^{n+1})\\
&+\frac{4\theta}{1+\theta}\frac{\ka_1}{(\nu^{n+1}-a)^2}\sum_{i\in L}(\ka_i-\nu^{n+1})^2f_i\frac{u_i^2}{u^2}+\frac{6\ka_1}{a}F.
\end{aligned}
\ee
In deriving the last inequality in (\ref{c2g12}) we have used that $\kappa_i>0$ for each $i.$
Now fix $\theta$ so that $\frac{8\theta}{1+\theta}=a^2,$
so we get the right hand side of (\ref{c2g9}) is strictly negative when provided $\ka_1>\frac{4}{a^2}$ which complete the proof.
\end{proof}

Let us assume that the flow exists in $[0,T)$ with $0<T<\infty$ such that the norm of $u^2(t), \forall t\in[0,T)$ is uniformly bounded in
$C^{2}(\Omega).$ Due to the concavity of $F,$ we can apply  the Evans-Krylov theorem \cite{CC95} to get uniform
$C^{2+\alpha}(\Omega)$ estimates which in turn will lead to $C^{2+\alpha,\frac{2+\alpha}{2}}(\Omega\times(0,T))$ estimates. And the long time existence follows by proving a priori estimates in any compact time interval for the corresponding norms.

In order to prove equation (\ref{Int16}) in Theorem \ref{Intt1}, according to Theorem \ref{c2gt0} , we only need to find a uniform bound $C$ which is independent of $T$ for $u|D^2u|$ on the boundary $\partial\Omega\times[0,\infty)$.

Following Lemma 3.4 in [LX10], we obtain that, for any fixed $x\in\ol\Omega_\epsilon:=\{x\in\ol\Omega, d(x,\partial\Omega)\leq\epsilon\},$
\[u(x,t)-u(x,0)\leq\int_0^\infty uw\F dt=u(x,t^*)\int_0^\infty w\F dt\leq C\epsilon,\]
which implies that,
\[\int_0^\infty w\F dt\leq C\;\mbox{in $\ol\Omega_{\epsilon}$}.\]
Therefore, by Lemma \ref{Prel0} and Corollary \ref{Grec0} we conclude that when $0<\epsilon\leq\epsilon_0,$ there exists a $\tilde{t}$ such that for any $t>\tilde{t},$ we have $0\leq F-\sigma<\delta\;\;\mbox{in $\ol\Omega_\epsilon$},$ where $\tilde{t}$ only depends on $\delta.$ Combining with Theorem \ref{C2bt0} and Theorem \ref{c2gt0} gives a uniform bound for $u|D^2u|.$

\section{Convergence to a stationary solution} \label{Con}
Let us go back to our original problem (\ref{Int2}), which is a scalar parabolic differential equation defined on the cylinder $\Omega_T=\Omega\times[0,T)$ with initial value $u(0)=u_0\in C^{\infty}(\Omega)\cap C^2(\ol\Omega)$ and $u_0|_{\partial\Omega}=0.$ In view of the a priori estimates, which we have estimated in the preceding sections, we know that
\be\label{Con0}
u|D^2u|\leq C,
\ee
\be\label{Con5}
\sqrt{1+|Du|^2}\leq C,
\ee
and hence
\be\label{Con1}
\mbox{$F$ is uniformly elliptic in $u$}.
\ee
Moreover, since $F$ is concave, we have uniform $C^{2+\alpha}(\Omega)$ estimates for $u^2(t),\;\forall t\geq 0.$ Thus the flow exists for all $t\in[0,\infty).$

By integrating equation (\ref{Int1}) with respect to $t$, we get
\be\label{Con2}
u(x,t^*)-u(x,0)=\int_0^{t^*} \F uw dt.
\ee
In particular,
\be\label{Con3}
\int_0^{\infty}\F uw dt<\infty\;\;\mbox{$\forall x\in\Omega.$}
\ee
Hence for any $x\in\Omega$ there exists a sequence $t_k\rightarrow\infty$ such that $\F u(x,t_k)\rightarrow 0.$

On the other hand, $u(x,\cdot)$ is monotone increasing and bounded (see Lemma 3.3 of \cite{LX10}). Therefore
\be\label{Con4}
\lim_{t\rightarrow\infty}u(x,t)=\tilde{u}(x)
\ee
exists, and is of class $C^{\infty}(\Omega)\cap C^1(\ol\Omega).$ Moreover, $\tilde{u}(x)$ is a stationary solution of our problem, i.e.,
$F\lll(\tilde{\Sigma}\rrr)=\sigma,$ where $\tilde{\Sigma}=\left\{(x,\tilde{u}(x))\mid x\in\Omega\right\}.$

\section{Uniqueness and foliation} \label{Unf}
\begin{theorem}\label{Unft0}
Suppose f satisfies (\ref{Int5})-(\ref{Int11}), in addition,
\be\label{Unf1}
\sum_if_i>\sum_i\lambda_i^2f_i\;\;\mbox{in $K\cap\{0<f<1\}$}.
\ee
Let $\Sigma_i=\{(x,u_i(x)\mid x\in\Omega\},\;i=1,2,$ be two graphs such that
\be\label{Unf0}
\sup_{x\in\Omega}f(\kappa[\Sigma_1])<f(\kappa[\Sigma_2]),
\ee
where $\Sigma_i\; i=1,2$ are strictly locally convex graphs (oriented up) in $\mathbb{H}^{n+1}$ over $\Omega\subset\mathbb{R}^n$ with the same boundary $\Gamma_\epsilon$ in the horosphere $P_{\epsilon}=\{x_{n+1}=\epsilon\}$ or with the same asymptotic boundary $\Gamma=\partial\Omega.$
Then there holds
\be\label{Unf2}
u_1>u_2, \mbox{in $\Omega$}.
\ee
\end{theorem}
\begin{proof}
We first observe that the weaker conclusion
\be\label{Unf3}
u_1\geq u_2
\ee
is as good as the strict inequality (\ref{Unf2}), in view of the maximum principle.

Hence prove by contradiction, assume (\ref{Unf3}) is not valid, in another word,
\be\label{Unf4}
E(u_2)=\{x\in\Omega:u_2(x)>u_1(x)\}\neq\emptyset.
\ee
Then there exists point $p_i\in\Sigma_i$ such that
\[0<d_0=d\lll(\Sigma_1,\Sigma_2\rrr)=d(p_1,p_2)=\sup_{p\in\Sigma_1}\{\inf_{q\in\Sigma_2\cap I^+(\Sigma_1)} d(p,q):(p,q)\in\Sigma_1\times\Sigma_2\},\]
where d is the distance function in $\mathbb{R}^{n+1}$, and $I^+(\Sigma_1)=\{(x,x_{n+1}):x^{n+1}\geq u_1(x)\}.$

Let $\chi$ be the maximal geodesic from $\Sigma_1$ to $\Sigma_2$ realizing this distance with end point $p_1$ and $p_2,$ and parametrized by arc length. Denote by $\bar{d}$ the distance function to $\Sigma_1,$
\[\bar{d}(q)=\inf_{p\in\Sigma_1}d(p,q).\]

Since $\chi$ is maximal, $\Upsilon=\{\chi(t):0\leq t<d_0\}$ contains no focal points of $\Sigma_1,$ hence there exists an open neighborhood $\mathfrak{U}=\mathfrak{U}(\Upsilon)$ such that $\bar{d}$ is smooth in $\mathfrak{U},$ and $\mathfrak{U}$ is a tubular neighborhood of $\Sigma_1,$ and hence covered by an associated normal Gaussian coordinates system $(x^{\alpha})$ satisfying $x^{n+1}=\bar{d}$ in $\{x^{n+1}>0\}.$

Now $\Sigma_1$ is the level set $\{\bar{d}=0,\}$ and the level set
\[\Sigma(s)=\{x\in\mathfrak{U}:\bar{d}=s\}\]
are smooth hypersurfaces.
Since the principle curvatures of $\Sigma(t)$ at points along the normal geodesic emanating from any point of $\Sigma_2$ (say near $p_2$) are given by ode
\[\ka'_i(s)=\ka_i^2-1.\]
hence by (\ref{Unf1}) we have
\be\label{Unf5}
\frac{d}{ds}f(\ka)(s)=\sum\ka^2_if_i-\sum f_i<0\;\;\mbox{in $K\cap\{0<f<1\}.$}
\ee

Next, in the same way, we consider a tubular neighborhood $\mathfrak{N}$ of $\Sigma_2$ with corresponding normal Gaussian coordinates $(x^\alpha).$ The lever sets
\[\tilde{\Sigma}(r)=\{x^{n+1}=r\},\;\;\mbox{$-\epsilon<r<0$,}\]
lies below $\Sigma_2=\tilde{\Sigma}(0)$ and are smooth for small $\epsilon.$

Since the geodesic $\chi$ is perpendicular to $\Sigma_2,$ it's also perpendicular to $\tilde{\Sigma}(r)$ and the length of the geodesic segment of $\chi$ is $-r.$ Hence we deduce
\[d\lll(\Sigma_1,\tilde{\Sigma}(r)\rrr)=d_0+r.\]
Further more, for fixed r, the hypersurface $\tilde{\Sigma}(r)$ touches $\Sigma(d_0+r)$ at $p_r=\chi(d_0+r)$ from below. The maximum principle then implies
\[f|_{\tilde{\Sigma}(r)}(p_r)\leq f|_{\Sigma(d_0+r)}(p_r)\]

On the other hand, $\tilde{\Sigma}_2(r)$ converges to $\Sigma_2$
It follows from (\ref{Unf5}) that
\[f(\kappa[\Sigma_2])(\chi(d_0))\leq f(\kappa[\Sigma_1])(\chi(0)).\]
It's a contradiction to (\ref{Unf0}).
\end{proof}

\section*{Acknowledgement}
The author would like to thank Professor Joel Spruck for his guidance and support.


\bigskip


\begin{thebibliography}{10}
\bibitem[A94]{A94} B.Andrews, Contraction of convex hypersurfaces in Euclidean space, {\em Calc.Var.PDE \bf{2}}, (1994), 151--171.
\bibitem[C89]{C89} Kung-Ching Chang, Heat flow and boundary value problem for harmonic maps,
{\em Annales de l'I. H. P.,} {section \bf{C}} (1989), Vol 6, 363-395.
\bibitem[CC95]{CC95}L. Caffarelli, and X. Cabré, Fully nonlinear elliptic equations, {\em American Mathematical Society},(1995).
\bibitem[CNS84]{CNS84}L. Caffarelli, L. Nirenberg and J. Spruck, The Dirichlet problem for nonlinear second-order elliptic equations I, {\em Comm. Pure Applied Math.} {\bf 37} (1984), 369--402.
\bibitem[CNS85]{CNS85} L. Caffarelli, L. Nirenberg and J. Sprick, The Dirichlet problem for nonlinear second-order elliptic equations III, {\em Acta Math.} {\bf{155}} (1985), 261--301.
\bibitem[CNS86]{CNS86} L. Caffarelli, L. Nirenberg and J. Sprick,Nonlinear second order elliptic equations IV, {\em Current Topics in P.D.E. Kinokunize Co.}, (1986),1--26.
\bibitem[G96]{G96} C.Gerhardt, Closed Weigngarten hypersurfaces in Riemannian manifolds, {\em J.Differential Geom.\bf{43}}, (1996), 612--641.
\bibitem[G06]{G06} C. Gerhardt, Curvature Problem, {\em Int.Press}, Somerville, MA, 2006.
\bibitem[GS08]{GS08} B. Guan, and J. Spruck, Hypersurfaces of constant curvature in Hyperbolic space II, preprint, arxiv.org/abs/0810.1781, {\em J. Eur. Math. Soc. to appear}.
\bibitem[GSZ09]{GSZ09} B. Guan, J. Spruck and M. Szapiel, Hypersurfaces of constant curvature in Hyperbolic space I, {\em J. Geom. Anal.} {\bf{19}} (2009), no. 4, 772--795.
\bibitem[GS10]{GS10} B.Guan and J.Spruck, Convex hypersurfaces of Constant curvature in hyperbolic space, preprint.
\bibitem[H75]{H75} Richard S. Hamilton, Hamonic maps of manifolds with boundary,
{\em Lecture notes in mathematics} {\bf{471}}, Springer (1975).
\bibitem[L96]{L96} Gray M. Lieberman,Second order parabolic differential equations, {\em World Scientific Pub Co}, (1996).
\bibitem[LX10]{LX10} L. Lin and L. Xiao, Modified mean curvature flow of star-shaped hypersurfaces in hyperbolic space, preprint. 









\end{thebibliography}
\end{document}